\documentclass[10pt]{amsart}
\usepackage{fancyhdr}
\usepackage{latexsym,amsfonts,amscd,yhmath,epsfig,epic}
\usepackage{ccaption}
\changecaptionwidth
\captionwidth{5.6in}

\usepackage{mathtools,amsmath,amssymb,mathrsfs,amscd,amsthm,amsfonts,graphicx,bbm}
\usepackage{eucal}                   %special calligraphic letters
\usepackage[colorlinks, plainpages]{hyperref}

\usepackage[cmtip,all]{xy} 	%for commutative diagrams
\usepackage{color}
\def \oR{{\overline {R}}}
\def \oX{{\overline {X}}}

\def \wX{{\widetilde {X}}}
\def \wR{{\widetilde {R}}}

\def\wt{\widetilde}
\def\ol{\overline}

\def \Rr{R^{red}}
\def \Xr{X^{red}}

\def\k{\mathbbm{k}}

\def \eps{\varepsilon}
\def \del{\delta}
\def \vp{\varphi}
\def \x{\langle x \rangle}
\def \sm{\smallsetminus}

\def\fm{\frak{m}} 
\def\fp{\frak{p}}
\def\fq{\frak{q}}

\newcommand{\R}{{\mathbb R}}

\newcommand{\Z}{{\mathbb Z}}

\newcommand{\C}{{\mathbb C}}
\newcommand{\Q}{{\mathbb Q}}

\newcommand{\K}{{\mathbb K}}   			%----> ground field
\newcommand{\kc}{{\mathcal C}}

\newcommand{\kf}{{\mathcal F}}

\newcommand{\km}{{\mathcal M}}
\newcommand{\kn}{{\mathcal N}}
\newcommand{\ko}{{\mathcal O}}

\newcommand{\kt}{{\mathcal T}}

\newcommand{\fn}{\mathfrak{n}}

\DeclareMathOperator{\Spec}{Spec}

\DeclareMathOperator{\Quot}{Quot}
\DeclareMathOperator{\Supp}{Supp}
\DeclareMathOperator{\Ann}{Ann}

\DeclareMathOperator{\NNor}{NNor}
\DeclareMathOperator{\Ker}{Ker}
\DeclareMathOperator{\NRed}{NRed}
\DeclareMathOperator{\Coker}{Coker}

\DeclareMathOperator{\nil}{nil}
\DeclareMathOperator{\INNS}{INNS}

\newtheorem{theorem}{Theorem}
\newtheorem{lemma}[theorem]{Lemma}
\newtheorem{corollary}[theorem]{Corollary}

\newtheorem{proposition}[theorem]{Proposition}
\newtheorem{definition}[theorem]{Definition}

\newtheorem{example}[theorem]{Example}
\newtheorem{remark}[theorem]{Remark}

\begin{document}
%\title {On Delta for Isolated Non-Normal Singularities}
\title[The Delta Invariant of isolated Non-Normal Singularities] {The Delta Invariant and Fiberwise Normalization for Families of isolated Non-Normal Singularities}
%----------Author 1
\author{Gert-Martin Greuel and Gerhard Pfister}
%\address{
%University of Kaiserslautern\\
%Department of Mathematics\\
%Erwin-Schroedinger Str.\\
%67663 Kaiserslautern\\
%Germany}

%\email{greuel@mathematik.uni-kl.de}
%\email{pfister@mathematik.uni-kl.de}
\maketitle
%\today

\begin{abstract}
We prove the semicontinuity of the delta invariant
in a family of schemes or analytic varieties with finitely many (not necessarily reduced) isolated non-normal singularities,  in particular for families of generically reduced curves.
We define and use a modified delta invariant for isolated non-normal singularities of any dimension that takes care of embedded points. Our results  generalize results by Teissier and Chiang-Hsieh--Lipman for families of reduced curve singularities. The base ring for our families can be an arbitrary PID such that our semicontinuity result provides possible improvements for algorithms to compute the genus of a curve.\\

\end{abstract}
\bigskip 

\bigskip
\tableofcontents

\section*{Introduction}
The delta invariant, also called genus defect, is an important numerical invariant of a singular reduced curve and is therefore often considered  for algebraic curves over the complex numbers, but also for curves over finite fields, e.g. in coding theory. The delta invariant was extended to generically reduced  complex analytic curves in \cite{BG90} 
and it was shown that  it can be used to control the topology in a family of such curves by taking care of the influence of embedded points. In \cite{Gr17} the delta invariant
was further extended to complex-analytic isolated non-normal singularities of any dimension and 
its behavior was studied in connection with simultaneous normalization.

The study of simultaneous normalization of deformations of a reduced curve singularity has been initiated by Teissier in the 1970’s in the complex analytic setting. The main result was, that a flat family of reduced curve singularities over a normal base space admits a simultaneous normalization if and only if the delta invariant of the curve singularities is locally constant.
This was further carried on by Chiang-Hsieh and Lipman \cite{CL06}
in the algebraic setting for families of reduced curves defined over a perfect field, clarifying some points in the proof given in \cite{Te78}. In \cite{CL06}
the authors get also intermediate results for families of higher dimensional reduced and pure dimensional varieties, but the $\del$-constant criterion for simultaneous normalization is only proved for families of reduced curves 
(and for projective morphisms with equidimensional reduced  fibers of arbitrary dimension, replacing the $\delta$-invariant by the Hilbert polynomial).

The results by Chiang-Hsieh and Lipman motivated us to reconsider the $\del$-constant criterion for families of schemes with finitely many isolated (not necessarily reduced) non-normal singularities, including the case
of generically reduced curves, defined over an arbitrary field. We define and use in the algebraic setting a modified delta invariant for an isolated non-normal singularity (INNS) of any dimension analogous to the complex analytic case, which coincides with the classical  $\del$-invariant for reduced singularities.
One of our main results are semicontinuity theorems for this new $\del$-invariant for families of schemes parametrized by the spectrum of a principal ideal domain (Theorem \ref{thm.semicont1} and its corollaries in arbitrary characteristic, Theorem \ref{thm.char0} in characteristic 0).  We like to emphasize that the semicontinuity of $\del$ holds for fibers over closed and non-closed points in a neighbourhood of a given point (in contrast to e.g. \cite[Proposition 3.3] {CL06}).

We apply the semicontinuity to prove  a $\delta$-constant criterion for fiberwise resp. simultaneous normalization (the two notions coincide e.g. in characteristic 0) of a family of INNSs. This means that a family of affine Noetherian schemes over the spectrum of an arbitrary PID with fibers having only finitely many isolated non-normal singularities and with singular locus finite over the base admits a simultaneous normalization if and only if the  $\del$-invariant of the fibers is constant (for a precise formulation see Theorem \ref{thm.simnormPID} and Corollary \ref{cor.simnormPID}).

Although we use ideas from \cite{CL06}, the proofs of our main results are quite different. In \cite[Theorem 4.1]{CL06} it is assumed (for the $\del$-constant criterion for simultaneous normalization of reduced curves) that the fibers are reduced and pure dimensional and that the base scheme is the spectrum of a complete, or Henselian, or analytic normal local ring. We do neither assume that the fibers are reduced nor that they are pure dimensional.
Our restriction (in the more general situation of families of INNSs) is that the base scheme is the spectrum of a PID. We conjecture that the results hold also for normal base spaces of any dimension, but  the non-reducedness of the fibers provides essential technical difficulties.

Since our base rings include $\Z$ and $\k[t]$, $\k$ any field, we just mention in passing that our results have interesting computational applications. E.g., if an isolated non-normal singularity is defined over $\Z$  resp. over $\k[t]$, the computation of the $\del$-invariant over $\Q$ resp. $\k(t)$ 
can be estimated and speeded up by the (much cheaper) computation modulo any (not only lucky) prime $p\in \Z$ resp. modulo $\langle t-a\rangle$, $a$ any element in $\k$.  This applies of course also to $\del$ for reduced curves and hence 
can be used to improve algorithms to compute the genus of a curve.
We refer to \cite[Remark 24]{GP21}, where we considered $\delta$ for families of parametrized curve singularities and to \cite{GPS21} for an algorithm, showing that the semicontinuity can lead to an impressive speed up of the calculations.

That we allow non-reduced singularities in the fibers is not an artificial assumption but occurs naturally in connetion with families of parametrized curves. Consider e.g. an analytic morphism $\phi:\C \times S \to \C^n \times S$ over $S$ such $\phi_s:\C \to \C^n$ is the parametrization of a reduced curve $C_s$ in $\C^n $ for $s\in S$. Let $X=\phi(\C \times S)$ be closed in  $\C^n \times S$ and flat over $S$. Then the fibers $X_s$ of $X \to S$ have in general non reduced singularities (the reduction of $X_s$ coincides with  $C_s$) and our results apply to this situation.
\smallskip

All rings in this paper are associative, commutative and with 1, ring maps map 1 to 1, and a ring map of local rings maps the maximal ideal to the maximal ideal. Moreover, we assume that all rings, modules, and schemes are Noetherian, without always explicitly stating this.
\medskip

{\bf Notation: }\label{not.>}
$A, R$ denote rings, $\k$ an arbitrary field, $\dim$ the Krull dimension and $\dim_\k$ the $\k$-vector space dimension.

If $\fp_1, ..., \fp_r$ are the minimal prime ideals of $R$, we denote by $\fp^i$ 
the intersection\footnote{\,The empty intersection is the whole ring $R$. E.g., if no minimal primes $\fp_j$ with $\dim R/\fp_j = i$ exist then $R^i=0$ and $X^i = \emptyset$.} of the $\fp_j$ with $\dim R/\fp_j=i$ and by  $\fp^{>i}$ the intersection
of the $\fp_j$ with $\dim R/\fp_j>i$. 
With $X= \Spec R$ and $\Xr = \Spec \Rr$, where $\Rr$ denotes the {\em reduction} of $R$, we define for $i \geq 0$:
\[
\begin{array}{cllllll}
 R^i &:= &R/\fp^i, \ &X^i &:=&\Spec R^i,\\
R^{>i}&:= &R/\fp^{>i}, \ &X^{>i}&:= &\Spec R^{>i}.
\end{array}
\]
Note that $R^i$  and $R^{>i}$ are reduced  and thus $X^i$  and $X^{>i}$ are reduced subschemes of $X$.
In particular, $X^0$ is a finite set of reduced, isolated points of $\Xr$
and $X^{>0} =\Xr$ iff $X$ has no isolated points.

We set $$r_i(X) := \sharp\{\text{irreducible componentes of } X^i\},$$
which is the number of $i$-dimensional irreducible components of $X$\,\footnote{\,By definition, the irreducible components of $X$ are the reduced schemes $\Spec R/\fp_j, j=1,...,r.$}.

If  $\vp: A\to R$ is a ring map, $\fp$ a prime ideal of $A$  and $k(\fp) = A_\fp/\fp A_\fp =Q(A/\fp)$ the residue field of $A$ at $\fp$, we set for an $R$-module $M$
$$M(\fp) := M_\fp\otimes_{A_\fp} k(\fp)=M \otimes_A k(\fp)$$  and call it the {\em fiber of $M$ over $\fp$} and call $R(\fp)$ denotes the fiber of $\vp$ over $\fp$.

Let $f=\Spec \vp: X = \Spec R \to \Spec A = S$ be the induced map of schemes and $t\in S$ the point corresponding to $\fp$. Then 
$$X_t := f^{-1}(t) := \Spec R(\fp)$$
denotes the fiber of $f$ over $t$. We set $f^i:=f\,|\,X^i$ resp. $f^{>i}:=f\,|\,X^{>i}$
and $(X^i)_t:=(f^i)^{-1}(t)$ resp. $(X^{>i})_t:=(f^{>i})^{-1}(t)$

\medskip

{\bf Acknowledgement:} We thank the reviewer for useful comments and Dmitry Kerner for his questions that helped improve the presentation.

\section{Delta for an Isolated Non-Normal Singularity}\label{sec.1}

Let $R$ be a reduced ring. Then $Q(R)$, the {\em total quotient ring of  $R$}, is a direct product of fields. If
$\fp_1,...,\fp_r$ are the minimal primes of $R$ then $Q(R)$ is the direct product of the fields  $Q(R/\fp_j) $.   $\oR$ denotes the {\em integral closure} of $R$ in $Q(R)$. $\oR$ or, more precisely, the natural inclusion  $R \hookrightarrow \overline R$  is called the {\em normalization of $R$}. $\oR$ is the direct product of the integral closures 
of $R/\fp_j$ in  $Q(R/\fp_j) $ (cf. \cite[Lemma 28.52.3]{Stack}, tag 035P).

If $R$ is not reduced, let $\pi: R \twoheadrightarrow R^{red}$ be  the natural projection and  
$\nil(R) := \ker(\pi)$   the ideal of nilpotent elements of  $R$.
%\end{align*}
We denote by $\nu^{red}: \Rr \hookrightarrow  \oR$ the normalization of $R^{red}$
and call $\oR$ or the composition
$$ \nu:= \nu^{red} \circ \pi: R \twoheadrightarrow R^{red} \hookrightarrow \overline R $$
 the {\em normalization} of $R$.
$R$ is called  {\em normal} if $\nu: R \to \overline R$ is an isomorphism.
This is equivalent to
$R_\fp$ being a normal domain for every prime ideal $\fp\subset R.$ We often write $\oR/\Rr$ in place of $\oR / \nu(R)$. 

For an arbitrary $R$-module let $\Ann_R(M)=\{g\in R\, | \, gM=0\}$ be the annihilator ideal of $M$ in $R$.

\begin{definition}\label{def.cond}
Let $R$ be a  ring.  We define 
\begin{enumerate} 
\item $\kc_R :=\Ann_R(\oR/\Rr)\subset R, \text { the {\em conductor ideal} of }  R,\\
 \wt{\kc}_R  := \kc_R \cap \Ann_R(\nil(R)) \subset R, \text { the {\em extended conductor ideal} of }  R,$\\
$C_R:= \Spec R/\kc_R$ the {\em conductor scheme} of $R$ and\\  
$\wt{ C}_R:= \Spec R/\wt{\kc}_R$ the {\em extended  conductor scheme} of   $R$.
\item The {\em non-normal locus of $R$} is denoted as
$$\NNor(R):= \{\fp \in \Spec R \, |\, R_\fp \text{ is not normal} \}.$$
It contains the {\em non-reduced locus}
$$\NRed(R):= \{\fp \in \Spec R \, |\, R_\fp \text{ is not reduced} \}.$$
 \end{enumerate}
\end{definition}

\begin{remark}\label{rm.nnor}\text{ }
{\em
\begin{enumerate}
\item $\kc_{R^{red}}= \Ann_{\Rr}(\oR/\Rr) = \pi(\kc_R) $ is the conductor ideal of $\Rr$. We have $\nil(R) \subset \kc_R$ since $\nil(R) = \ker(\nu)$, and $\pi$ induces an isomorphism  $R/\kc_R \xrightarrow {\cong} \Rr/\kc_{R^{red}}$.

\item We have always $\NNor(R)\subset V(\wt{\kc}_R)$ but equality may not hold and $\NNor(R)$ may not be closed in $\Spec R$. However, if $\oR/\Rr$ is (module-) finite over $R$ (equivalently, $\oR$ is finite over $R$), then 
$\Supp_R(\oR/\Rr)$  coincides with $V(\kc_R)$ and is therefore closed in $\Spec R$. 
  
\item  We get: If $\oR$ is finite over $R$ (for examples see Remark \ref{rm.nagata}), then the {\em non-normal locus} of $R$
\begin{align*}
\NNor(R) &= \NRed(R) \cup \NNor(R^{red})\\  
               &= V( \Ann_R(\nil(R))) \cup V(\kc_R) \\
               &= V( \wt{\kc}_R).
\end{align*}
is the zero-locus of the extended conductor ideal $\wt {\kc}_R$ and hence closed in $\Spec R$. 
\end{enumerate}
}
\end{remark}

 \begin{definition} \label{def.inns} We say that $\fp \in \Spec R$ is an {\em isolated non-normal point} of $R$, or that {\em $R$ has an isolated non-normal singularity}  at $\fp$, 
if $\fp$ is an isolated point of $\Spec R$ or an isolated point of $\NNor(R)$.
 We also say that $\fp$ is an $\INNS$ (of $R$ or of $\Spec R$). Furthermore we set
$$ \INNS(R) := \{\fp \in \Spec R \mid \fp \text { is an  {\em INNS }of } R\},$$
the locus of isolated non-normal points of $R$.
 \end{definition}
 
 We note that $\fp$ is an $\INNS$ if either $R_\fp$ is not normal (and there is an open neighbourhood $U$ of $\fp$ with $R_\fq$ normal for all $\fq \in U\smallsetminus \{\fp\}$), or $R_\fp$  is normal and $\fp$ is an isolated reduced point of $\Spec R$. We include  isolated reduced points in our definition of $\INNS$, since these play a special role in our definition of the delta invariant (Definition \ref{def.delta} and Remark \ref{rm.eps}).
Isolated singularities are  $\INNS$, with typical examples (generically) reduced curves.
 
 {\begin{remark}\label{rem.inns}
{\em Using the notations from the introduction we have 
$$V(\fp^0)=\{\fp \in \Spec R \mid \fp \text{ is a reduced isolated point of } \Spec R\}.$$
If $\oR$ is  finite over $R$,  then 
$\fp \in\INNS(R)$ iff  $\fp \in V(\fp^0)$ or $\fp$ is an isolated point of  $V(\wt{\kc}_R)$. Therefore 
 every $\fp \in \INNS(R)$ is a closed point of  $\Spec R$, i.e. a maximal ideal.}
 \end{remark}
 \smallskip
 
\begin{lemma}\label{lem.nnor}
Let $\oR$ be finite over $R$. Then $\INNS(R)$ is a finite set if and only if $R/\wt{\kc}_R$ is an Artinian $R$-module.
\end{lemma}
\begin{proof}
$V(\fp^0)$ is finite and we have $\NNor(R) =V(\wt{\kc}_R)= \Supp_R(R/\wt{\kc}_R)$.
The result follows since $R/\wt{\kc}_R$ is Artinian 
$\Leftrightarrow$ $\Spec R/\wt{\kc}_R$ is a finite set (\cite[Prop. 8.3]{AM69}). 
\end{proof}
 
\medskip
 
We are now going to define the delta and the epsilon invariant of a local ring. 
 Let $\fm \in \Spec R$ be a maximal ideal of $R$ and $M$ an $R$-module.  The {\em $0$-th local cohomology group} of $M$ is the submodule
$$H^0_\fm(M) = \{x\in M\ | \ \fm^k x = 0 \text { for some } k \geq 0\}.$$
Since $M$ is Noetherian, $H^0_\fm(M)$ is Noetherian too and is annihilated by some power of $\fm$; hence $H^0_\fm(M)$ has finite length, i.e. is Artinian.
\smallskip

\begin{definition} \label{def.delta}
Let $(R,\fm)$ be a  local  ring with normalization $\oR$, $\k$ a field, and $\k \to R$ a ring map. We define:

\begin{itemize} 
\item [(i)]  the {\em  epsilon invariant} of $R$ (w.r.t. $\k$), 
$$\varepsilon_\k(R) := \dim_\k H^0_\fm( R),$$
\item [(ii)]the {\em  delta invariant} of $R$ (w.r.t. $\k$),
$$\delta_\k (R):=  \dim_\k \overline R/R^{red} - \eps_\k(R),$$
\item [(iii)] the {\em (multiplicity of the) conductor} of $R$ (w.r.t. $\k$) 
$$c_\k (R):=  \dim_\k \oR/\kc_{R^{red}}- \eps_\k(R).$$ 
\end{itemize}
\end{definition}
Hence, if $R$ is reduced and $\dim R >0$ then $\eps_\k(R)=0$ and $\delta_\k (R)=  \dim_\k \overline R/R$, the usual definition of $\delta_\k$.

\begin{remark}\label{rm.eps}
{\em
Let   $K=R/\fm$ denote the residue field of the local ring $(R,\fm)$ and assume that $\dim_\k K < \infty$. 
\begin{enumerate}
\item $\eps_\k(R)$ is always finite while $\delta_\k(R)$ and $c_\k (R)$ may be infinite. If  $R$ is an INNS with $\oR$ finite over $R$, then $\delta_\k(R)$ and $c_\k (R)$ are also finite  (Lemma \ref{lem.artin}).

\item  If  $\dim R = 0$, then  $\oR=R^{red} = K$,
 $\nil(R) = \fm$ and $H^0_{\fm}(R) =R$.  
We get  
 $\delta_\k (R) = c_\k(R) = - \eps_\k(R) = - \dim_\k R = -\dim_\k \nil(R) -\dim_\k K <0$. \\
 In particular, $\del$ and $\eps$ are never 0 and $\delta_\k (R)=- \eps_\k(R) = -\dim_\k K $  if $R$ is a reduced (hence normal) isolated point.
 
\item Let  $\dim R > 0$ and let $R$ be an INNS with $\oR$ finite over $R$. Since $R_\fp$ is reduced for $\fp \in U  \smallsetminus \fm$, $U$ some open neighbourhood of $\fm$ in $\Spec R$, we have $\nil(R) =  H^0_{\fm}(R)$ and   $\eps_\k(R) = \dim_\k \nil(R)$. 

\item Let  $\dim R > 0$.  If $R$ is normal then $\delta_\k(R)=c_\k (R)=\eps_\k(R)=0$. 
If $R$ is reduced, then $R$ is normal $\iff$ $\delta_\k(R)=0$ $\iff$  $c_\k(R)=0$ (for the last equivalence see Lemma \ref{lem.artin}(2)). 
But if $R$ is not reduced, then $\delta_\k(R)=0$ may happen for non-normal $R$ (see Example \ref{ex.delta} (3)).
\end{enumerate}
 }
\end{remark}

\begin{example}\label{ex.delta}
{\em \begin{enumerate}
\item The ideal 
$I=\langle x\rangle \cap \langle x^2,y^2,xy\rangle = \langle x^2,xy\rangle \subset \k[[x,y]]$
defines
a line with embedded component. With  $R=\k[[x,y]]/I$ we get $\delta_\k(\Rr) = 0$ and $\eps_\k(R) = 1$, hence $\delta_\k(R) = -1$ and $c_\k (R) =-1$.

\item The ideal 
$$\quad\quad I=\langle x^3y+x^2y^2,x^2y^2+xy^3\rangle = 
\langle x+y\rangle \cap \langle x\rangle \cap \langle y\rangle \cap \langle x^2,y^3\rangle \subset \k[[x,y]],$$ 
defines 3 lines with an embedded component at $0$. For $R = \k[[x,y]]/I$ we have $\delta_\k(\Rr) = 3$ and $\eps_\k(R) = \dim_\k \sqrt I/I =1$\,\footnote{\, We compute $\eps$ and $\del$ with {\sc Singular} \cite{DGPS}:  {\tt codim} computes
$ \dim_\k \nil(R)$ = $\dim_\k \sqrt I/I$ and the procedure {\tt normal(..,"wd")} computes $\delta_\k(\Rr)$;
the number of isolated points of $\Spec R$ can be determined with a primary decomposition of $I$.} and hence $\delta_\k(R) = 2$. Since $\Rr$ is a reduced plane curve singularity, we get
$c_\k (\Rr)=  2\delta_\k (\Rr)=6 $ and $c_\k (R)=c_\k (\Rr)- \eps_\k(R) = 5.$

\item $I= \langle z,x^2-y^3 \rangle \cap \langle x, y, z^2 \rangle$, $R = \k[[x,y]]/I$, defines a cusp in the $(x,y)$-plane and an embedded point in the $z$-direction. Then 
$\delta_\k(\Rr) = 1$ and $\eps_\k(R)  =1$ and hence $\delta_\k(R) = 0$.
\end{enumerate}
}
\end{example}

\begin{lemma}\label{lem.fields}
 Let $(R,\fm)$ be a local ring, $K= R/\fm$, $\k \to R$ a ring map, and $M \neq 0$ a finitely generated $R$-module. Then $ \dim_\k M <\infty$ $\Leftrightarrow$ $M$ is Artinian and $\dim_\k K < \infty$.
 \end{lemma}
 
\begin{proof} If $ \dim_\k M <\infty$ then $M$ is Artinian since it satisfies obviously the descending chain condition.  By Nakayama's lemma,
$M/\fm M$ is a finite dimensional $K$-vector space $ \neq 0$. We have $\dim_\k K \leq  \dim_\k M/\fm M\leq \dim_\k M < \infty$.
 Conversely, if  $M$ Artinian then $\fm^n M=0$ for some $n$.  The $K$-vector space $\fm^kM/\fm^{k+1}M$ has finite $K$-dimension, hence finite $\k$-dimension since  $\dim_\k K < \infty$. Thus $ \dim_\k M <\infty$. 
\end{proof}

\begin{lemma}\label{lem.artin}
 Let $(R,\fm,K)$  be a  local  ring with normalization $\oR$ finite over $R$. 
 \begin{enumerate} 
 \item   The following are equivalent:
 \begin{enumerate}
 \item[(i)]  $\fm$ is an $\INNS$; 
  \item[(ii)] $R/\wt{\kc}_R$ is an Artinian $R$-module;
 \item[(iii)] $\oR/\kc_{R^{red}}$ and $\nil(R)$ are Artinian  $R$-modules;
 \item[(iv)] $\oR/R^{red}$ and $\nil(R)$ are Artinian $R$-modules.
 \end{enumerate}  
 \item  Let $\k$ be a field and $\k \to R$ a ring map. Then the following are equivalent:
   \begin{enumerate}
 \item[(i)]  $\fm$ is an $\INNS$ and $\dim_\k K < \infty$; 
  \item[(ii)] $\dim_\k (R/\wt{\kc}_R)$ is finite, and $\dim_\k K < \infty$ if $R=K$;
  \item[(iii)] $c_\k (\Rr)$ and $\eps_\k(R)$ are finite. 
  \item[(iv)] $\delta_\k(\Rr)$ and $\eps_\k(R)$ are finite; 
\end{enumerate}  
If any of these conditions hold, $c_\k (R)$ and $\delta_\k(R)$ are finite and satisfy 
 $$c_\k (R) = \delta_\k(R) + \dim_\k(R/\kc_R).$$
 \end{enumerate}
\end{lemma}
 
\begin{proof} 
If $\Spec R$ is a reduced point (i.e., $R=K$) then the statements (i),...,(iv) from (1)  all hold. Moreover, the statements (i),...,(iv) from (2) are equivalent since $\dim_\k (R/\wt{\kc}_R) = c_\k (\Rr) =\delta_\k(\Rr)=0$ and $\eps_\k(R)<\infty \iff \dim_\k K < \infty$. 
If $R$ is normal of dimension $>0$ none of the statements (i),...,(iv) from (1) and (2) hold.
We may thus assume that $R$ is not normal. 

(1) It is well known that  a finitely generated $R$-module $M\neq 0$ is Artinian $\Leftrightarrow$ $\fm^k M=0$ for some $k>0$ $\Leftrightarrow$ $\dim M=0$ $\Leftrightarrow$ $\Supp_R (M) =\{\fm\}$. Now (1) follows from
 $\NNor(R) = \Supp_R(R/\wt{\kc}_R) = \Supp_R(\nil(R)) \cup \Supp_R(\oR/\Rr)$.

(2) The equivalence of (i) - (iv) follows from (1) for $\k=K$, noting that $\eps_\k(R) = \dim_\k \nil(R)$ if $\dim(R) >0$ and $\dim_\k \nil(R) = \eps_\k(R)-1$  if $\dim(R) =0$.
Together with Lemma \ref{lem.fields} the  equivalence follows for arbitrary $K$.
The exact sequence
$$ 0 \to \Rr/\kc_{R^{red}} \to \oR/\kc_{R^{red}}  \to \oR/R^{red} \to 0$$
implies $c_\k (\Rr) = \delta_\k(\Rr) + \dim_\k(\Rr/\kc_{R^{red}} )$ and hence $c_\k (R) = \delta_\k(R) + \dim_\k(R/\kc_{R} )$ by definition of $c_\k$ and $\delta_\k$.
\end{proof}

Now let  $R$ be a not necessarily local ring with $\oR$ finite over $R$. 
Since $\wt{\kc}_{R\fp}=(\wt{\kc}_R)_\fp$
it follows from Lemma \ref{lem.nnor} and \ref{lem.artin} that 
$R$ has only finitely many non-normal points  $\iff$ $R/\wt{\kc}_R$ is Artinian.
If $R$ is a $\k$-algebra then
$\dim_\k (R/\wt{\kc}_R) < \infty$  $\iff$  $\NNor(R)$ is finite and 
$\dim_\k k(\fp) < \infty$ for all $\fp \in \NNor(R)$ and this implies the finiteness of $\del$, $\eps$ and $c$ ad $\fp$.
\begin{definition}\label{def.deltaglob}
Let $R$ be a $\k$-algebra with normalization $\oR$ finite over $R$. Assume that
 $R$ has only finitely many isolated non-normal points 
and that $\dim_\k k(\fp)< \infty$ for all $\fp\in \INNS(R).$
%=\{\fp_1,...,\fp_r\}$$. 
%$\delta_\k (R_{\fp_i})< \infty$ for $i=1,...,r$. 
We define 
$$\delta_\k (R) := \sum_{\fp \, \in \, \INNS(R)}\delta_\k (R_{\fp}),$$
 $\eps_\k (R) := \sum_\fp \eps_\k (R_{\fp})$ and $c_\k (R) := \sum_\fp c_\k (R_\fp)$  ($\fp$ runs through $\INNS(R)$), which are all finite.
\end{definition}
Note that every isolated point of $\Spec R$ (reduced or not) counts in the above sum.

\begin{example}{\em
Let $R_\C=\C[x,y]/ I$, $I=\langle y^2-2x^2 \rangle \cap \langle y-x^2\rangle$. 
 $V(I)$ consists of two straight lines and a parabola meeting in $(0,0)$ and in $(\pm\sqrt 2,2)$. The three INNS correspond to the maximal ideals $\fp, \pm\fq$.
 $\fp$ is a triple point  with $\del_\C(R_{\C,\fp})=3$,
 while $\pm\fq$ are ordinary nodes with $\del_\C(R_{\C,\pm\fq}) = 1$ each, hence $\del_\C(R_\C)=5.$ 
 
Let $R_\Q=\Q[x,y]/ I$,  with $I$ as above. Then $R_\Q$ has (in $\Spec R_\Q$) two INNS, at  the maximal ideals $\fp= \langle x,y\rangle$ and 
$\fq= \langle x^2-2,y-2\rangle$, with $k(\fp)=\Q$ and $k(\fq)=\Q(\sqrt 2).$
We get $\del_\Q (R_{\Q,\fp}) =3$ and $\del_\Q (R_{\Q,\fq}) =2$, hence
$\del_\Q (R_\Q) =5.$
The equality $\del_\Q (R_\Q)= \del_\C (R_\C)$ is a general fact, since\,\footnote{\,Let $B$ be a $\k$-algebra and $K$ a separable field extension of $\k$ then $\overline{B\otimes_\k K}=\overline B\otimes_\k K$ (\cite{Stack}, Lemma 32.27.4, tag 0C3N) and hence 
$\delta_K(B\otimes_\k K)=\delta_\k (B).$
}
$R_{\mathbb C}=R_{\mathbb Q}\otimes_{\mathbb Q}\mathbb C$.
}
\end{example}
\medskip

The following interpretation of $\del$  as an Euler characteristic is useful.
 Using that $R^{red}=R^0 \oplus R^{>0}$,  $\oR=\oR^0 \oplus \oR^{>0}$ and  $R^0=\oR^0$ we get
\begin{align*}
\eps_\k(R) &= \dim_\k \Ker(R\to \oR )+\dim_\k(R_0),\\
                  & = \dim_\k \Ker(R\to  \oR^{>0} )\\
\delta_\k (R) &=  \dim_\k \overline R^{>0}/R^{>0} - \eps_\k(R)\\
                 &= \dim_\k \Coker(R \to \oR^{>0})  - \dim_\k \Ker(R\to \oR^{>0} ).
 \end{align*}

\begin{lemma} \label{lem.euler}
With the \ref{not.>} assumptions of Definition \ref{def.deltaglob}
 consider the $2$--term complex with $R$ in degree 0,
$$R^\bullet : 0\to R \to \oR^{>0}\to 0.$$
Then 
\begin{align*}
\del_\k(R)
= - \chi_{_\k} (R^\bullet),
\end{align*}
where $\chi_{_{\k}} (L^\bullet):= \sum_i  (-1)^i \dim_\k H^i(L^\bullet)$ for a complex $L^\bullet$ of $\k$-modules with finite dimensional cohomology.
\end{lemma}

%%%%%%%%%%%%%%%%%%%%%%%
The following technical lemma compares $\del$ and $\eps$ of $R$ with that of a finite modification of $R$ whose positive dimensional part is a partial normalization of $R^{>0}$. It is a key lemma for the semicontinuity of $\del$.

\begin{lemma}\label{lem.diff}  
Let $R$ be a  $\k$-algebra with  $\oR$
finite over $R$, having only finitely many isolated non-normal singularities, with residue fields finite over $\k$.
 Consider a finite morphism of $\k$-algebras
  $\mu: R \to \wR$. 
Let  $N\subset \Spec R$ be a finite set of closed points with residue fields finite over $\k$, such that 
$\Spec \mu$ is an isomorphism over $ \Spec R\sm N$.

 Then the positive dimensional parts $R^{>0}$ and $\wR^{>0}$ have the same normalization and $\mu$ satisfies
$$
\begin{array}{cllllll}
\dim_\k \Coker(\mu) - \dim_\k \Ker(\mu)& = &\delta_\k(R) - \delta_\k(\wR)\\
& = & \varepsilon_\k(\wR)-\varepsilon_\k(R)+\dim_\k \Coker(\mu^{>0})
\end{array}
$$
with $\mu^{>0}: R^{>0} \to \wR^{>0}$ the induced map. $\mu^{>0}$ is finite and injective and a partial normalization\footnote{\,Let $\nu: R \to \oR$ be the normalization of $R$. A  {\em partial normalization} of $R$  is a birational morphism $\mu :R \to \wR$ such that $\nu= \tilde \nu \circ \mu:R\to \wR \to \oR$, with $\tilde \nu$ the normalization of $\wR$.} 
 of the reduced positive dimensional part of $R$\,\footnote{\, In the case $R^{>0}=0$, i.e. $X^{>0}=\emptyset$, the statements here and in the following are to be interpreted accordingly, e.g. with $\del_\k(R^{>0})=0$ and $\eps_\k(R^{>0})=0$.} and all numbers are finite. 
\end{lemma}

\begin{proof} By assumption $R$ has only finitely many non-normal singularities with $\del_\k(R)$ and $\eps_\k(R)$ finite (Lemma \ref{lem.artin}). Since 
$\mu$ is an isomorphism outside finitely many closed points, $\Ker (\mu)$ and $\Coker(\mu)$ are Artinian. Then the kernel and cokernel of $\mu$ and of $\mu^{>0}$ are finite over $\k$  (Lemma \ref{lem.fields}).

We have  $\Spec \widetilde{R}^{red}=\Spec \widetilde{R}^{>0} \,\cup$ \{finitely many isolated points\} and 
 the restriction of $\mu$  induces a birational\,\footnote{\,A morphism of 
schemes is {\em birational} if it is a bijection between the generic points and an isomorphism of the corresponding local rings. A morphism of rings is birational if this holds for the corresponding morphism of schemes.}
 morphism
$R^{>0} \to \wR^{>0}$, since it is an isomorphism outside finitely many closed points.
Let $\nu^{>0}:R^{>0} \to \oR^{>0}$ be the normalization of $R^{>0}$. 
By \cite[Lemma 28.52.5 (3), tag 035Q]{Stack} $\nu^{>0}$ factors as
 $\nu^{>0}=\tilde \nu\circ \mu^{>0}: R^{>0} \to \wR^{>0} \to \oR^{>0}$ 
with $\tilde \nu: \wR^{>0} \to \oR^{>0}$ the normalization of $\wR^{>0}$
and $\tilde \nu$ finite. Hence $ \mu^{>0}$ is a partial normalization. It is finite since $\mu$ is finite and injective since $R^{>0}$ is reduced.

 Now consider the $2$--term complexes (with $R$ resp. $\wR$ in degree 0)
\[
\begin{array}{l}
R^\bullet : 0\to R \to \oR^{>0}\to 0 \, , \\
\wR^\bullet : 0\to \wR \to \oR^{>0}\to 0
\end{array}
\]
and the morphism of complexes $\mu^\bullet: R^\bullet\to \wR^\bullet$ with 
$\mu^0=\mu$ and  the identity in degree 1.
Let $K^\bullet$ resp. $C^\bullet$ be the $1$--term complexes $\Ker(\mu)$ resp. $\Coker (\mu)$, concentrated in degree $0$. Then we have the exact sequence of complexes
\[
0\to K^\bullet\to R^\bullet\to \wR^\bullet\to C^\bullet\to 0\, .
\]
Taking Euler characteristics
we get (by Lemma \ref{lem.euler})
$$\dim_\k \Coker(\mu)-\dim_\k \Ker(\mu) =
\chi_{_\k}(\wR^\bullet)-\chi_{_\k}(R^\bullet)=\delta_\k(R)-\delta_\k(\wR)$$ 
showing the first equality. 
Since $\delta_\k(R)=\delta_\k(R^{>0})-\varepsilon_\k(R)$ we get
\[
\delta_\k(R)-\delta_\k(\wR)=\varepsilon_\k(\wR)-\varepsilon_\k(R)+\delta_\k(R^{>0})-\delta_\k(\wR^{>0}).
\]
From this and from the inclusions $R^{>0} \hookrightarrow \wR^{>0} \hookrightarrow \oR^{>0}$ the second equality follows.
\end {proof}

\begin{remark}\label{rm.nagata}
{\em \begin{enumerate}
\item  For all results of this paper we have to assume that $\oR$ is (module-) finite over $R$. Integral domains that satisfy this conditions are called {\em N-1 rings}. An {\em N-2 ring} (or {\em Japanese ring})  is
an integral domain  $R$ such for every finite field extension $L$ of $Q(R)$
the integral closure of $R$ in $L$ is finite over $R$.
$R$ is a {\em Nagata ring} if $R$ is Noetherian and for every prime ideal $\fp$ the ring $R/\fp$ is N-2 (see \cite[Lemma 10.157.2, tag 03GH]{Stack}). 
Hence $\oR$ is finite over $R$ if $R$ is Nagata.
\smallskip

\item  $R$ is  Nagata iff (cf. \cite [1.4.3]{CL06})\\
 (a) for every maximal ideal $\fn$ of $R$ the canonical map $R_\fn \to \widehat R_\fn$ from the local ring $R_\fn $ to its completion is reduced (flat with reduced fibers) and\\ 
 (b) for every reduced finitely generated $R$-algebra $R'$ the set of normal points is open and dense in $\Spec R'$.

Condition (b) is implied by (a) if $R$ is semi-local.
For further properties of Nagata rings we refer to \cite[Section 10.157, tag 032E]{Stack}. 
\smallskip

\item  Examples of Nagata rings are:
\begin{enumerate}
\item fields, $\Z$,  complete local Noetherian  rings,
\item Dedekind domains with perfect\footnote{\,A field $\k$ is perfect if $\k$ is of characteristic 0 or of characteristic $p>0$ and every element has a $p$-th root (e.g. if $\k$ is finite).}
fraction field\footnote{\,This statement is formulated in \cite{Stack} only for Dedekind domains with fraction fields of characteristic 0, but the proof works  for perfect fields of positive characteristic as well.}, 
\item finite type ring extensions of any of the above, \\
(for (a) (b) (c) see \cite[Proposition 10.157.16, tag 0335]{Stack}),
\item quasi-excellent, in particular excellent rings (e.g. analytic local rings),
 (\cite[Lemma 15.51.5, tag 07QV]{Stack}), 
\item localizations of a Nagata ring 
(\cite[Lemma 10.157.6, tag 032U]{Stack}),
\item $A$-algebras (essentially) of finite type over a  Nagata ring $A$ 
(\cite[Proposition 10.157.16, tag 0335]{Stack}),
\item  $ A[[x_1,...,x_n]]$  is Nagata if $A$ is Nagata,
(\cite[Appendix A, Property PSEP]{KS19}).
\end{enumerate}
A scheme $X $ is called Nagata if for every $x \in X$ there exists an affine open neighbourhood  $U \subset X$ of $x$ such that the ring $\ko_X(U)$ is Nagata. Note that there are discrete valuation rings that are not Nagata (\cite[Example 10.157.17, tag 09E1]{Stack}).
\end{enumerate}
}
\end{remark}
%===================================

\section{Semicontinuity of the Delta Invariant}\label{sec.2}
We consider now families of isolated non-normal singularities over a principal ideal domain. Recall that a principal ideal domain $A$ is a regular 1-dimensional domain with $A_\fp$ a discrete valuation ring for $\fp \in \Spec A$ and with $\fp$ a maximal ideal for $\fp\neq \langle 0 \rangle$. For us the most important examples are $\Z$ and $\k[t]$.
The following Proposition \ref{prop.fibers} is fundamental for the semicontinuity results of this paper.

\begin{proposition}\label{prop.fibers} 
 Let $\vp: A\to R$ be a flat morphism of  rings with  $A$ a principal ideal domain and
 $\mu: R \to \wR$ a finite morphism of $A$-algebras.
Assume that
\begin{enumerate}
\item the composition $\widetilde \vp:= \mu\circ \vp: A \to \wR$ is flat,
\item  $\Ker(\mu)$ and $\Coker(\mu)$ are finite over $A$,
\item the normalization $\overline {R(\fq)}$
is finite over $R(\fq)$ and  the residue fields at the non-normal points of $R(\fq)$ are finite over $k(\fq)$ for $\fq \in  Im (\Spec \vp)$.
\end{enumerate}
 Then, for $\fp \in Im (\Spec \vp)$ there exists an open neighborhood $U\subset \Spec A$ of $\fp$ such that for $\fq \in U\cap  Im (\Spec \vp)$ the following holds:
 \[
\begin{array}{clll}
(i) & \delta_{k(\fp)}(R(\fp))-\delta_{k(\fq)}(R(\fq))=\delta_{k(\fp)}(\wR(\fp))-\delta_{k(\fq)}(\wR(\fq)).\\
(ii) & \eps_{k(\fp)}(R(\fp))-\eps_{k(\fq)}(R(\fq)) = \eps_{k(\fp)}(\wR(\fp))-\eps_{k(\fq)}(\wR(\fq))\\
&+\dim_{k(\fp)} \Coker (\mu(\fp)^{>0})- \dim_{k(\fq)} \Coker (\mu(\fq)^{>0}).\\
(iii) &  \text{If } \Spec \wR^1 \cap \Spec \wR^{>1} = \emptyset \text{ (e.g. if $\wR =\oR$ or $\wR^1 =0$), then}\\
 & \eps_{k(\fp)}(R(\fp))-\eps_{k(\fq)}(R(\fq)) = \eps_{k(\fp)}(\wR(\fp))-\eps_{k(\fq)}(\wR(\fq))\\ 
 &+\dim_{k(\fp)}(\wR^{>1}/ \mu(R^{>1}))\otimes_A k(\fp)- \dim_{k(\fq)}(\wR^{>1}/ \mu(R^{>1}))\otimes_A k(\fq)\\
 &\geq \eps_{k(\fp)}(\wR(\fp))-\eps_{k(\fq)}(\wR(\fq)).\\
(iv) & \text{If } \Ker(\mu)=0 \text{ then } \Ker(\mu(\fq): R(\fq)\to \wR(\fq)) =0 \text{ for } \fq \neq \fp. \\
 \end{array}\\
\]
Here   $\wR(\fq)= \wR \otimes_Ak(\fq)$ and  $\mu(\fq)^{>0}: R(\fq)^{>0} \to \wR(\fp)^{>0}$ is the induced map of positive dimensional parts, which is a partial normalization of $R(\fq)^{>0}$.
\end{proposition}

\begin{proof} 

We set
\[
\begin{array}{lcl}
\kn: & = & \Ker (\mu: R\to \wR) ,\\
\km: & = & \Coker (\mu: R\to \wR) .
\end{array}
\]
Both $R$-modules are finitely generated $A$-modules by assumption and hence $\kn(\fq)=\kn\otimes_A k(\fq)$ and $\km(\fq)=\km\otimes_A k(\fq)$ are finite dimensional vector spaces over $k(\fq)$ for $\fq \in \Spec A$. Then they are Artinian $R(\fq)$-modules with 
$$N(\fq) := \Supp_{R(\fq)} \kn(\fq) \cup \Supp_{R(\fq)} \km(\fq)$$
 a finite set of closed points of $R(\fq)$. 
 The set
$N:= \Supp_R \kn \cup \Supp_R \km$
is closed in $\Spec R$ with  $ N(\fq) = \{\fn \in N | \fn\cap A=\fq \} = N\cap \Spec R(\fq)$.
Since  $\Spec \mu: \Spec \wR \sm \mu^{-1}(N) \to \Spec R \sm N$
is an isomorphism,
the fiber map $\Spec \mu(\fq): \Spec \wR(\fq) \to \Spec R(\fq)$  is an isomorphism over  $\Spec R(\fq) \smallsetminus N(\fq)$  with
 $\mu(\fq)^{-1}(N(\fq))$
a finite set of closed points \cite[Lemma 36.39.1.,tag 02LS]{Stack}  (since $\mu$ and hence $\mu(\fq)$ is finite).
It follows that  the assumptions of Lemma \ref{lem.diff} are satisfied for $\mu(\fq): R(\fq) \to \wR(\fq)$ and $\k =k(\fq)$. 
In particular, $\mu(\fq): R(\fq)^{>0} \to \wR(\fq)^{>0}$ is a partial normalization of $ R(\fq)^{>0}$.

Let $\fq \in Im(\Spec \vp)$ be non-zero. Since $A$ is principal, $\fq$ is a maximal ideal, generated by one element $t_\fq \in A$. We denote the image of  $t_\fq$ in $R$ resp. $\wR$ by $f_\fq$ resp. $\tilde f_\fq$, which are non-zero divisors since $R$ and
$\wR$ are flat over $A$.
Consider  the commutative diagram 
\[
\xymatrix{0 \ar[r] & R\ar[d]^\mu\ar[r]^{f_\fq} & R\ar[r]\ar[d]^\mu & R(\fq)\ar[r]\ar[d]^{\mu(\fq)}\ar[r] & 0\\
0\ar[r] & \wR\ar[r]^{\tilde f_\fq} &\wR\ar[r] &  \wR(\fq)\ar[r] & 0,
}
\]
with exact rows. 
 Since $A$ is a principal ideal domain 
we have a decomposition
\[
\km=\kf\oplus\kt
\]
with $\kf$ a free $A$--module and $\kt$ an $A$--torsion submodule concentrated on finitely many maximal ideals in $A$.
$\kn$ is  a free $A$-module (since it is torsion free as a submodule of the flat, hence torsion free $A$-module $R$). Since $\kf$ and $\kn$ are free, 
they are of constant rank $m$ and $n$ respectively and  we get for every $\fq \in Im(\Spec \vp)$, 
\[
m= \dim_{k(q)}\kf(\fq), \ n= \dim_{k(q)}\kn(\fq).
\]

Now fix a  $\langle 0 \rangle \neq \fp \in Im(\Spec \vp)$. There exists an open neighbourhood $U$ of $\fp$ in $\Spec A$ such that $\kt_\fq=0$ for $\fq \in U \smallsetminus \{\fp\}$ and hence $\dim_{k(\fp)}\kt_\fp<\infty$. The snake lemma, applied to the diagram above, gives for $\fq \in U$ the exact sequence
\[
0 \to \kn \xrightarrow{f_\fq} \kn \to \Ker(\mu(\fq)) \to \km \xrightarrow{\tilde{f_\fq}} \km\to \Coker (\mu (\fq))\to 0,\\
\]
and from this we get
\[
\begin{array}{c}
0\to \kn(\fq)\to \Ker(\mu(\fq))\to \Ker(\tilde f_\fq)\to 0,\\[0.5ex]
0\to \Ker(\tilde f_\fq)\to \kf \oplus \kt \xrightarrow{\tilde f_\fq}
\kf \oplus \kt \to \Coker(\mu(\fq))\to 0\ .
\end{array}
\]
$\tilde f_\fq$ respects the decomposition into free and  torsion part, with $\Ker(\tilde f_\fq|\kf)=0$ and $\Coker(\tilde f_\fq|\kf)=\kf(\fq)$.
Since $\kt$ is finite dimensional, kernel and cokernel of $\tilde f_\fq: \kt \to \kt$ have the same dimension for each $\fq \in U$ (being 0 for $\fq \neq \fp$). 

If $\kn=0$ then $\Ker(\mu(\fq)) = \Ker(\tilde f_\fq)$ and  $=0$ for $\fq \neq \fp$ since $\kt_\fq =0$ and statement (iv) follows. 
\medskip

By Lemma \ref{lem.diff} 
 $\Coker(\mu(\fq))$ and $\Ker(\mu(\fq))$ are finite dimensional over $k(\fq)$ and we get 
\[
\begin{array}{ll}
m &= \dim_{k(\fq)} \Coker(\mu(\fq))-\dim_{k(\fq)} \Ker(\tilde f_\fq)\\
&= \dim_{k(\fq)} \Coker(\mu(\fq)) - \dim_{k(\fq)} \Ker(\mu(\fq)) +n.
\end{array}
\]
It follows that
$ \dim_{k(\fq)} \Coker(\mu(\fq)) - \dim_{k(\fq)} \Ker(\mu(\fq)) = m-n$
is independent $\fq \in U \sm  \langle 0 \rangle $. The same holds for $\fq =\langle 0 \rangle$ since $\kt(\fq) = 0$ and hence $\Coker(\mu(\fq)) = 
\kf(\fq)$ and $\Ker(\mu(\fq))=\kn(\fq)$.
Lemma \ref{lem.diff} implies now statement (i) and (ii).
\medskip

To prove (iii) assume that $\Spec \wR^1 \cap \Spec \wR^{>1} = \emptyset$.
Then $\wR(\fp)^0=(\wR^1)(\fp)$ and $\wR(\fp)^{>0}=(\wR^{>1})(\fp)$
for $\fp \in Im (\Spec \vp)$ ($\vp$ is flat) and $\Coker (\mu(\fp)^{>0}) = (\wR^{>1}/ \mu(R^{>1}))\otimes_A k(\fp)$.
By assumption $\wR/ \mu(R)$ is a finite $A$-module and hence also $\wR^{>1}/ \mu(R^{>1})$. Thus
$\dim_{k(\fp)}(\wR^{>1}/ \mu(R^{>1}))\otimes k(\fp)$ is semicontinuous on $\Spec A$
(\cite[Lemma 1]{GP20}), which proves 
$\dim_{k(\fp)} \Coker (\mu(\fp)^{>0}) \geq \dim_{k(\fq)} \Coker (\mu(\fq)^{>0})$
and hence (iii).
We notice, that if $\wR^{>1}=0$ the proof gives 
$ \eps_{k(\fp)}(R(\fp))-\eps_{k(\fq)}(R(\fq)) = \eps_{k(\fp)}(\wR(\fp))-\eps_{k(\fq)}(\wR(\fq))$.
\end {proof}

\begin{lemma}\label{lem.flat}
 Let $\vp: A \to R$  be a morphism of rings with $A$  a principal ideal domain. Let the normalization $\nu: R \to \oR$ be finite and
 $\mu :R \to \wR$ be a finite morphism, 
  which is a partial normalization of $R$. 
 \begin{enumerate}
 \item Let $\vp$ be flat.
 \begin{enumerate}
 \item  [(i)] Let $Q$ be the (non-empty) intersection of some associated primes of $R$ and set $R':=R/Q$. Then the induced map  $\vp ': A \to R'$ is flat. In particular $\vp^{red}: A \to R^{red}$ is flat.
  \item [(ii)] The map $\tilde{\vp} = \mu \circ \vp : A \to \wR$ is flat  
if $\wR$ is reduced. In particular, $\ol{\vp} = \nu \circ \vp : A \to \oR$ is flat.
  \end{enumerate}
\item  Let $R$ and $\wR$ be reduced. Then $\vp$ is flat $\iff$ 
 $\tilde{\vp}$ is flat.
   \item Let $\fn \in \Spec R, \, \fp = \fn \cap A$ and $\vp: A_\fp \to R_\fn$ flat.     
 \begin{enumerate}  
 \item [(i)]  If $R_\fn(\fp)= R_\fn \otimes_{A_\fp} k(\fp)$ is reduced, then $R_\fn$ is reduced.
  
\item [(ii)]  If $\dim R_\fn \geq 2$ and $R_\fn(\fp)$ reduced, then 
{\em depth}$(R_\fn)\geq 2$ and $r_1(R_\fn) = 0$  (i.e., no 1-dimensional irreducible component of $\Spec R$ passes through $\fn$). 
\item  [(iii)] If  $\dim R_\fn \geq 2$ and $R_\fn(\fp)$ an INNS then: 
 $R_\fn(\fp)$ is reduced $\iff$ {\em depth}$(R_\fn) \geq 2$.
\end{enumerate}
 
 \item Let $\vp$ be flat and $\dim R/\fn \geq 2$ for every minimal prime $\fn$ of $R$. Assume that $R(\fp)$ has only isolated non-normal singularities for $\fp \in Im(\Spec \vp)$. 
 Then the following are equivalent:
 \begin{enumerate} 
\item [(i)] $R$ is reduced, 
\item [(ii)] for each $\fp \in  Im(\Spec \vp)$, $R(\fp)$ is reduced at all normal closed points of $R$.
 \end{enumerate}
 \end{enumerate}
\end{lemma}

\begin{proof} Since $A$ is a PID, $\vp$ is flat $\iff$ $\vp(a)$ is a non-zero divisor (n.z.d.) in $R$ for each $a\neq 0$ in $A$ $\iff$ $\vp(a)$ is not contained in any associated prime of $R$.

(1) The above characterization implies (i).
To see (ii), let $\nu= \tilde \nu \circ \mu:R\to \wR \to \oR$, with $\tilde \nu $ the normalization of $\wR$.
Consider first the case that $\mu =\nu: R \to \oR$ is the normalization of $R$.
Let $\fp_1,...,\fp_r$ be the minimal primes of $R$. Then $\oR = \oplus_i \ol{R/\fp_i}$ and $\bar\vp(a) = (b_1,...,b_r)$, with $b_i = \bar\vp(a \ mod \ \fp_i)$. Since $\vp^{red}$ is flat by (i), 
$\vp(a) \notin \fp_i$ for all $i$ and hence $b_i \neq 0$ for all $i$, showing that $\bar \vp(a)$ is a n.z.d. in $\oR$, i.e., $\bar \vp$ is flat. 
If $\wR$ is a partial normalization of $R$ and reduced, then $\wR \subset \oR$ and hence  $ \tilde{\vp} $ is flat.
  
(2) By (1) (ii) the flatness of $\vp$ implies that of $\tilde \vp$. The flatness of  $\tilde \vp$ implies that of  $\vp$ since $R$ is reduced and $R\subset \wR$.

(3) (i) By \cite[Corollary to Theorem 23.9]{Mat86} $R_\fn$ is reduced (normal),  if $ R_\fn \otimes_{A_\fp}k(\fq)$ is reduced (normal) for all $\fq \in \Spec A_\fp$. Since $A_\fp$ has only two prime ideals 
$R_\fn$ is reduced by the following Lemma \ref{lem.red}.

(ii) $R_\fn(\fp)$ is reduced iff it satisfies Serre's condition $(R_0)$ and $(S_1)$. \\
If $\dim R_\fn \geq 2$ and $R_\fn(\fp)$ reduced, then $\dim R_\fn(\fp) \geq 1$ and depth $R_\fn(\fp)\geq 1$ (from $(S_1)$).  Hence depth$(R_\fn) \geq 2$ since $\vp$ is flat (\cite[Corollary to Theorem 23.3]{Mat86}) and then $r_1(R_\fn) =0$ (\cite[Proposition 1.2.13]{BH98}). 

(iii) Implication $\Rightarrow$ follows direclty  from (ii). For the converse direction we note that $\nil(R) = H^0_{\fm}(R)$ for $(R,\fm)$  a local INNS of dimension $\geq 1$ (Remark \ref{rm.eps}) and that $H^0_{\fm}(R)=0$ iff depth$(R) \geq 1$ (\cite[Proposition 3.5.4.]{BH98}). Hence $R_\fn(\fp)$ is reduced iff 
$\text{depth}(R_\fn(\fp)) \geq 1$, which holds since $\fp$ is generated by a non-zero divisor and therefore
$\text{depth}(R_\fn(\fp)) =\text{depth}(R_\fn)-1$ (\cite{Stack} Lemma 10.71.7, tag 090R).

(4) Let $R$ be reduced and $\fn \in \Spec R$ a normal closed point of $R$. By assumption $\dim R_\fn \geq 2$ and by Serre's condition (S2) depth$(R_\fn) \geq 2$. Hence $R(\fp)_\fn = R_\fn(\fp)$
is reduced by (3) (iii), and this proves (i) $\Rightarrow$ (ii).
Conversely, as $R$ is reduced at all normal points we consider
 a non-normal point $\fn$ of $R$ ($\fn$ is then a closed point). Let $\fp = \fn \cap A$, $t$ a generator of $\fp$ and $f$ the image of $t$ in $R$. Since  $R_\fn$ is an INNS, $\nil(R_\fn)$ is concentrated on $\fn$ and hence killed by a power of $f$. Since each power of $f$ is a non-zero divisor of $R_\fn$, $\nil(R_\fn) =0$ 
and (ii) $\Rightarrow$ (i) follows.
\end{proof}

\begin{example}\label{ex.charp}
{\em
The condition in Lemma \ref{lem.flat} (3)(iii) that $R_\fn(\fp)$ is an $\INNS$ is necessary:\\
Let $A=\k[z], \k$ algebraically closed and char$(\k) = p >0$ and $R = A[x,y]/\langle f\rangle$, $f = y^p - x^p -z$. 
Then $R$ is regular of depth 2 at every closed point, and the canonical map $\vp: A \to R$ is flat ($z$ is a non-zero divisor in $R$). 
All fibers $R(s) = \k[x,y]/y^p-x^p -s$  over closed points $\langle z-s\rangle, s \in  \k$, are not reduced ($y^p-x^p -s= (y-x-s^{1/p})^p$) and not an INNS (the generic fiber is however regular).

Such an example is not possible in characteristic 0 (see Section \ref{sec.4}). E.g. if $f:(\C^n,0)\to (\C,0)$ is flat then the fibers $f^{-1}(t)$ are smooth for $t\neq 0$ close to $0$.}
\end{example}

\begin{lemma}\label{lem.red}
Let $\vp: (A,\fm) \to (R,\fn)$ be a flat morphism of  local rings.
\begin{enumerate}
\item $\oR \otimes_A Q(A) = \ol{R \otimes_A Q(A)}$; in particular, 
if $R$ is normal, then  $R \otimes_A Q(A)$ is normal.
\item If $(A,\fm)$ is a discrete valuation ring and $R \otimes_A A/\fm$ reduced, then $R\otimes_A Q(A)$ is reduced.
\end{enumerate}
\end{lemma}

\begin{proof} $A\subset R$ since  $R$ is flat, hence faithfully flat over $A$.

(1) We have  $R \otimes_A Q(A)= \{\frac{r}{a} |\, r\in R, a\in A \text{ a non-zero divisor}\}$, hence
$Q(R \otimes_A Q(A))= Q(R)$ and
$Q(\oR \otimes_A Q(A))= Q(\oR) = Q(R)$. Thus 
$\oR \otimes_A Q(A))$ is normal and
$\ol{R \otimes_A Q(A)} \subset \oR \otimes_A Q(A))$. Since the last inclusion is birational it is an equality. 

(2) Since $A$ is a DVR, $\fm  = \langle t \rangle$ for some $t$.
Then  $R\otimes_A Q(A) = R_t = \{\frac{r}{t^\nu} | r \in R ,\nu \geq 0\}$. 
Assume $(\frac{r}{t^\nu})^n =0$ for some $n$. Then $r^n =0$ since $t$ is a non-zero divisor of $R$ and $\bar r ^n=0$, $\bar r$ the image of $r$ in $R/\fm R = R/ tR$. Since $R/ tR$ is reduced, $r = t r'$ for some $r' \in R$. By induction $r \in \cap t^\nu R$, and $\cap t^\nu R = 0$ by Krull's intersection theorem. Hence $r=0$.
\end{proof}

The following lemma is used for a geometric interpretation of the technical assumption  that   
$\Coker(\nu)$ and $\Ker (\nu^{>1})$ are finite over $A$ in the proof of our main Theorem \ref{thm.semicont1}.

\begin{lemma} \label{lem.ann}
Let $A\to R$ be a ring map with $A$ Noetherian and $M$ a finitely generated $R$-module.  Then 
the following are equivalent.
\begin{enumerate}
\item [(i)] $M$  is finite over  $A$. 
\item [(ii)] $R/\Ann_R(M)$ is finite over $A$.
\item [(iii)] $R/I$ is finite over $A$ for every ideal $I\subset R$ with $\sqrt I = \sqrt{\Ann_R(M)}$.
\end{enumerate}
\end{lemma}

\begin{proof}
(i)$\Rightarrow$(ii): Let $M$ be generated over $R$ by $m_1,...,m_n$. Then
we can embedd $R/\Ann_R(M)$ in $M^n$ by $r \mapsto (rm_1,...,rm_n)$.  As $M$ is finite over $A$,  $M^n$ is finite over $A$ and since $A$ is Noetherian, $M^n$ is a Noetherian $A$-module. Since any submodule of a Noetherian module is Noetherian,   
$R/\Ann_R(M)$ is a Noetherian $A$-module and therefore finite over $A$. 

(ii)$\Rightarrow$(iii): If $I\subset J$ are two ideals in $R$ then $R/I$ finite over $A$ implies obviously that $R/J$ is finite over $A$. Moreover,
$$R/I^n  \text{ finite over $A $ for all } n\geq 1 \iff  R/I \text{ finite over } A,$$
To see ''$\Leftarrow$''  we consider the exact sequence
$0\to I^{n-1}/I^n \to R/I^n \to R/I^{n-1} \to 0$. Starting with $n=2$ we may assume by induction  that $R/I^{i-1}$ is finite over $A$ for $i=2,...,n-1$. Since $I^{n-1}/I^n$ is finite over $R/I$, it is finite over $A$. Hence $R/I^n$ is finite over $A$.

By assumption $R/\Ann_R(M)$ is finite over $A$ and hence $R/\Ann_R(M)^n$ is finite over $A$ for all $n$. There exists an $n$ with $\Ann_R(M)^n \subset I$ and therefore $R/I$ is finite over $A$.

(iii)$\Rightarrow$(i):  $M$ is finite over $R/\Ann_R(M)$ and since $R/\Ann_R(M)$ finite over $A$
by assumption, $M$ is finite over $A$.
\end{proof}

For  Lemma \ref{lem.fin} we introduce the following notations. 
Let   $\nu=\nu^{red} \circ \pi :R \to \Rr \to \oR$ be the  normalization map and $\nu^{>1}$  be the composition  
$$\nu^{>1}:R \twoheadrightarrow \Rr \hookrightarrow \oR= \ol{R^1}\oplus  \ol{R^{>1}}
\twoheadrightarrow \ol{R^{>1}}.$$
Let $\tilde \fp_i= \fp_i/\nil(R)$ be the minimal prime ideals of $R^{red}$   and $\tilde \fp^1$ resp. 
$\tilde\fp^{>1}$ be the intersection of the $\tilde \fp_j$ with $\dim R/\fp_j =1$ resp. $\dim R/\fp_j >1$.   We assume that  $\dim R/\fp_j \geq 1$ for all $j$.
Then $\nil(R) = \fp^1 \cap \fp^{>1}$ and the kernel of 
$$R^{red}  \hookrightarrow R^{red}/\tilde\fp^1\oplus R^{red}/\tilde\fp^{>1} \to \ol{R^{red}/\tilde \fp^{>1} }= \ol{R^{>1}}$$
is $\tilde\fp^{>1}$ with $\Ann_{R^{red}} (\tilde\fp^{>1}) = \tilde\fp^1$. It follows
$$\Ker(\nu^{>1}) = \fp^{>1} \text{ and }\Ann_{R} (\tilde\fp^{>1}) = \fp^1.$$
From the exact sequence
$0\to \nil(R) \to \fp^{>1} \to  \tilde\fp^{>1}\to 0$ we get $\Supp_R(\fp^{>1}) = \Supp_R(\nil(R)) \cup \Supp_R(\tilde\fp^{>1}))$ and therefore
$$\sqrt{\Ann_R (\fp^{>1})} = \sqrt{\Ann_R(\nil(R))} \cap \sqrt{\Ann_R(\tilde\fp^{>1})}=\sqrt{\Ann_R(\nil(R)) \cap \fp^1}.$$
Hence
$$\Supp_R(\Ker(\nu^{>1})) = \Supp_R(\nil(R)) \cup V(\fp^1),$$
$$\Supp_R(\Coker(\nu)) \cup \Supp_R(\Ker(\nu^{>1})) = V(\wt{\kc}_R ) \cup  V(\fp^1),$$
with $\wt{\kc}_R $ the extended conductor ideal.

Note that  $V(\wt\kc_R)=\NNor(R)$ resp.  $V(\fp^1)$ are the  non-normal locus resp.  the 1-dimensional part of $\Spec R$.

\begin{lemma} \label{lem.fin}   Let $A\to R$ is a ring map and let all minimal primes of $R$ have dimension $\geq1$.
 With the above notations  the following are equivalent.
\begin{enumerate}
\item [(i)] $\Coker(\nu)$  and $\Ker (\nu^{>1})$ are finite over $A$.
\item  [(ii)]  $R/\wt\kc_R \cap \fp^1$  is finite over $A$.
\item  [(iii)] $R/\wt{\kc}_R$  and $R/\fp^1$ are finite over $A$.
\end{enumerate}
\end{lemma}

\begin{proof} From the canonical exact sequence, for ideals $I,J \in R$,
$$0 \to R/I\cap J \to R/I\oplus R/J \to R/I+J \to 0,$$
(with $f \mapsto (f,f)$ and $(f,g) \mapsto f-g$) it follows that 
$$ R/I, \ R/J \text{ are finite over } A \iff
R /I\cap J\text{ is finite over } A$$ 
(the finiteness of $R/I+J $  follows from that of  $R /I\cap J$). This shows the equivalence of (ii) and (iii).
\medskip

We apply now Lemma \ref{lem.ann}. 
Since 
$\Ann_R(\Coker(\nu))=\kc_R$ (the conductor ideal of $R$),  $\Coker(\nu)$ is finite over $A$ iff
$R/\kc_R$ is finite over $A$, and since 
$$\sqrt{\Ann_R(\Ker(\nu^{>1}))} =\sqrt{\Ann_R(\nil(R)) \cap \fp^1},$$
$\Ker(\nu^{>1})$ is finite over $A$ iff $R/(\Ann_R(\nil(R)) \cap \fp^1)$ is finite over $A$. It follows that 
 $\Coker(\nu)$  and $\Ker (\nu^{>1})$ are finite over $A$ iff 
 $$R/\kc \cap \Ann_R(\nil(R)) \cap \fp^1 = R/\wt\kc_R \cap \fp^1$$ 
 is finite over $A$, showing the equivalence of (i) and (ii).
\end{proof}

We are now going to prove the semicontinuity of $\del$ and $\eps$ in flat families $A\to R$ over a principal ideal domain $A$. For a geometric interpretation of the assumption  that   $\Coker(\nu)$ and $\Ker (\nu^{>1})$ are finite over $A$ see Lemma \ref{lem.fin}.

\begin{theorem}\label{thm.semicont1}
Let $\vp:A\to R$ be a flat morphism of rings,  $A$  a principal ideal domain, and let the  normalization $\nu :R \to \oR$ be finite.
Let $X =\Spec R$,  $\oX =\Spec \oR$,
$n =\Spec \nu :\oX \to X$ and $f =\Spec \vp : X \to S=\Spec A$.
Assume that   $\Coker(\nu)$ 
and $\Ker (\nu^{>1}: R\to \ol{R^{>1}})$
are finite over $A$ and,
moreover, that for each $s\in f(X)$ the normalization of $X_s =f^{-1}(s)$
is finite over $X_s$ and that $X_s$ is normal outside finitely many isolated non-normal singularities at which the residue fields are finite over $k(s)$.

Then $\ol{ f} := f\circ n : \oX \to S$ is flat,  $\del_{k(s)}(X_s) < \infty, \,\eps_{k(s)}(X_s) < \infty$, and for each  $s\in f(X)$ there exists an open neighbourhood $V\subset S$ of $s$ such that  the following holds for $U=V\cap f(X)$: 
\begin{enumerate}
 \item 
$\delta_{k(s)}(X_s)-\delta_{k(t)}(X_{t})\\ 
\text{  } =  \delta_{k(s)}((X^{red})_s)-\delta_{k(t)}((X^{red})_{t})\\
\text{  } =\delta_{k(s)}((X^{>1})_s)-\delta_{k(t)}((X^{>1})_{t})\\
\text{  } =\delta_{k(s)}((\oX)_s)-\delta_{k(t)}((\oX)_{t})\\
\text{  } = \delta_{k(s)}((\oX^{>1})_s)-\delta_{k(t)}((\oX^{>1})_{t})$ for $t\in U$.
\item If $(\ol{X^{>1}})_t$  is normal for $t \in U\sm\{s\}$, then\\
$ \delta_{k(s)}(X_s)-\delta_{k(t)}(X_{t}) 
% = \delta_{k(s)}((\oX)_s)+r_1(X)
=\delta_{k(s)}((\ol{X^{>1}})_s)  \geq 0. $
\item $\delta_{k(s)}(X_s)-\delta_{k(\eta)}(X_{\eta}) =\delta_{k(s)}((\ol{X^{>1}})_s) \geq 0$, $\eta$ the generic point of $S$.
\item $\varepsilon_{k(s)} (X_s)-\varepsilon_{k(t)}(X_{t})= \eps_{k(s)}(({X}^{>1})_s)\geq 0$ for $t\in U\sm \{s\}$.
\item $(\ol{X^{>1}})_t = (\oX^{>1})_t:=(\ol{f} | \oX^{>1})^{-1}(t)$  is reduced for every $t\in U$. If $X$ is reduced then $X_t$ is reduced for $t\in U\sm \{s\}$.\\
\end{enumerate}
\end{theorem}

\begin{remark}\label{rem.fiber}{\em
Let $x_1,...,x_{r_s}$ be the  isolated non-normal singularities of $X_s$, with $s$ corresponding to 
a prime ideal $\fp \subset A$. Then each $x_i$ is a closed point of $X_s$ (cf. Remark \ref{rem.inns}), corresponding to a maximal ideal $\fn_i$ of $R(\fp)$ and
$$\del_{k(s)}(X_s)= \sum_1^{r_s}\del_{k(s)}(X_s,x_i),$$
with $\del_{k(s)}(X_s,x_i) := \del_{k(\fp)}(R(\fp)_{\fn_i}) 
= \dim_{k(\fp)}\ol{R(\fp)_{\fn_i}}/R(\fp)_{\fn_i}^{red} -\eps_{k(\fp)} (R(\fp)_{\fn_i}). $ 
If $s$ is a closed point of $S$ (i.e. $\fp\neq \langle 0 \rangle$) then the $x_i$ are closed points of $X$ and
$R(\fp)_{\fn_i}=R_{\fn_i}/\fp R_{\fn_i}$.
If $\fp = \langle 0 \rangle$ is the generic point $\eta$, then $k(\eta) = Q(A)$ and 
$R(\fp)_{\fn_i}=R_{\fn_i}\otimes_A Q(A)$. For a concrete example see Example \ref{ex.semicont1}.
}
\end{remark}

Before giving the proof, we'd like to comment on the result.  A semicontinuity theorem in the algebraic setting was proved by Chiang-Hsieh and Lipman in \cite[Proposition 3.3 and Theorem 4.1]{CL06}  under several assumptions, including the following (in our notation):
(i) $(A,\fm)$ is a normal local ring with perfect residue field, (ii) $A$ is complete, or $A$ is henselian and $R$ is a localization of a finitely generated $A$-algebra,
or  $A$ and $R$ are both analytic local rings,
%The natural map from $A$ to its completion $\hat A$ is normal,  
(iii)  $R$ is a formally equidimensional Nagata ring, and (iv) the special fiber $X_s = \Spec R/\fm R$ is a reduced curve and every closed point of $X$ is contained in $X_s$.
%the fibers  $X_t$ and $(\oX)_t$,  $t\in \Spec A$, are equidimensional and reduced. 
The authors prove that   $\delta_{k(t)}(X_t)-\delta_{k(\eta)}(X_{\eta}) =\delta_{k(t)}((\ol{X})_t)$ for $t\in \Spec A$ (since $\Spec A$ has only one closed point, the semicontinuity holds only for generalizations to the generic point $\eta$).
 Apart from the fact that there is no restriction for $\dim A$ in \cite{CL06}, our result is stronger in several ways.  We do not assume (a) that the residue fields at closed points of $\Spec A$ are perfect, (b) that $X$ is equidimensional, (c) that
 the fibers are reduced curves and (d) that $A$ is local; it can be e.g. $\Z$ or $\k[t]$, $\k$ an arbitrary field. Thus our semicontinuity holds also for closed points in a neighbourhood of the given point $s$ in $\Spec A$.
 
For complex analytic map germs $f:(X,x) \to (S,0)$, with $(S,0)=(\C,0)$
and $(X_0,x)$ a {\em reduced curve singularity,} the result is classical and due to Teissier \cite{Te78}. 
% \medskip
 
\begin{remark}\label{rem.semicont1}
{\em 
\begin{enumerate} 
\item  We have $\oR = \oR^1\oplus \oR^{>1}$ and
$\nil(R)=\Ker(R\to R^{red})=\Ker(\nu) \subset \Ker(\nu^{>1})$ 
since $\oR \to \oR^{>1}$ is surjective. Hence, if  $\Ker(\nu^{>1})$ is  finite over $A$ then $\Ker(\nu)$ is also finite over $A$. On the other hand, if $\Coker(\nu)$ is finite over $A$, then $\Coker(\nu^{>1})$ is finite over $A$
since  $\Coker(\nu)$ surjects onto $\Coker(\nu^{>1})$. 
\medskip

\item  We remark that every irreducible component of $\oX$ has dimension $\geq1$ ($\ol{f}$ is flat)  and that $\oX= \oX^{>1}\sqcup  \oX^{1}$ and 
$(\oX)_t= (\oX^{>1})_t\sqcup  (\oX^{1})_t$ with $(\oX^{>1})_t = (\oX)_t^{>0}$ and $(\oX^{1})_t = (\oX)_t^{0}$.
 In particular,
 \begin{align*}
\text{    } \quad (\ol{X^{>1}})_t = (\ol{X})_t
& \iff (\ol{X})_t \text{ has no isolated points}\\
  & \iff X \text{ has  no 1-dimensional components meeting } X_t.
\end{align*}
\item  Since  $\varphi$ is injective, the generic point $\eta =\langle 0\rangle $ is contained in $f(X)$  (\cite[Lemma 29.4.,tag 00FJ]{Stack}). It may however happen that $f(X)$ does not contain any open subset of $S$. E.g. $f(X) = \{\eta\}$ for  $A=\Z \to R = \Q[x]$ since $A\cap \fp = \langle 0\rangle $ for
every prime ideal $\fp \in R$
(see also Example \ref{ex.open}). 
 \medskip

\item  The statement of the theorem is especially interesting if $f$ is surjective or if $f$ is open (which holds if $X$ is of finite presentation over $S$ by \cite [Proposition 10.40.8.,tag 00I1]{Stack}, or for analytic maps).
 \end{enumerate}
 }
\end{remark}

\begin{proof}(of Theorem \ref{thm.semicont1})
Let  $\mu: R\to \wR$ be one of the maps  $\nu^{red}: R\to R^{red}$,
$\nu{'}:R\to R^{>1}$, $\nu:R\to \oR$, and
$\nu^{>1}: R\to \oR^{>1}.$ 
Since $\Ker (\nu^{red}) = \Ker (\nu) \subset \Ker( \nu^{>1}) = \Ker (\nu')$,
it follows that $\Ker (\mu)$ and $\Coker (\mu)$ are finite over $A$, since $\Ker (\nu^{>1})$ and $\Coker (\nu)$ are finite over $A$  by assumption.

If $R^{>1} =0$, i.e. $X$ is of pure dimension 1, then $R=\Ker (\nu^{>1})$ and hence $R$ is finite over $A$. Since $R$ is $A$-flat, it is torsion free, hence free and
$\dim_{k(\fp)}(R(\fp)) = \eps_{k(\fp)}(R(\fp)) = - \del_{k(\fp)}(R(\fp))$ is constant on $\Spec A$. 
Thus the theorem holds in this case trivially and we will assume in the following the $R^{>1} \neq0$.

By Lemma \ref{lem.flat} the maps 
$\tilde f= \Spec (\mu \circ \vp): \wX=\Spec \wR \to S$ are flat. We use the notation $(\wX)_t = \tilde f^{-1}(t)$.

(1) We can apply Proposition \ref{prop.fibers}  (i) to each map 
$\mu$ and we get the equalities in statement (1).

Now consider $\ol{f}^{>1}:= \ol{f} | \oX^{>1} :  \oX^{>1} \to S$. From Proposition \ref{prop.fibers}  (i) and (iii) we get
an open neighbourhood $U$ of $s$ such that for $t \in U\cap f(X)$ 
\begin{enumerate} 
\item [(a)]$\delta_{k(s)}({X}_s)-\delta_{k(t)}(X_{t})
=\delta_{k(s)}(({\ol X}^{>1})_s)-\delta_{k(t)}((\ol X^{>1})_t)$. 
\end{enumerate} 
Altogether this proves (1).

(2) $(\oX^{>1})_t$ normal implies $\delta_{k(t)}((\oX^{>1})_t) = 0$, $t\neq s$. Since 
$\ol X^{>1} =\ol {X^{>1}}$ is normal and of dimension $\geq 2$ at each closed point $x$, depth$(\ol X^{>1}) \geq 2$ at $x$. By Lemma \ref{lem.flat} (3)(iii),  $(\oX^{>1})_t, t\in f(X),$ is reduced at every closed point, hence at every point. Thus
 $\delta_{k(t)}((\oX^{>1})_s) \geq 0$ and (a) proves (2).

(3) By Lemma \ref{lem.red}
the ring $\oR^{>1}(\eta)$ and thus the generic fiber $\oX^{>1}_\eta$ is normal, and (3) follows from (2). 

To get estimates for $\eps$ we apply Proposition \ref{prop.fibers}  (iii) to 
 $f' = \Spec(\nu{'}\circ \vp) : X^{>1} \to S$  and get
\begin{enumerate} 
\item [(b)]$\eps_{k(s)}(X_s)-\eps_{k(t)}(X_t) = \eps_{k(s)}(({X}^{>1})_s)-\eps_{k(t)}(({X}^{>1})_t)\\
+\dim_{k(s)} \Coker(\nu'^{>1})\otimes_Ak(s) - \dim_{k(t)} \Coker(\nu'^{>1})\otimes_Ak(t)\\
= \eps_{k(s)}(({X}^{>1})_s)-\eps_{k(t)}(({X}^{>1})_t),$\\
since $\Coker(\nu'^{>1}) = R^{>1} / \nu'(R^{>1}) =0.$
\end{enumerate} 

(4) Since ${X}^{>1}$ is reduced,  $(X^{>1})_t$ is reduced  for $t \in U\sm \{s\}$ by (5). Hence $\eps_{k(t)}(({X}^{>1})_t)=0$ and (4) follows from (b).

(5) If $X$ is reduced, i.e. if $R$ is reduced, then $\nu:R\to \oR$ is injective. By 
Proposition \ref{prop.fibers} (iv) $R(\fq) \to \oR(\fq)$ is injective for  $\fq \in U\sm \{\fp\}$, $U$ some neighbourhood of $\fp$. Hence $X_t$ is reduced
for $t \in U\sm \{s\}$.
By the proof of (2) $(\oX^{>1})_t$ is reduced for all $t\in U$.
\end{proof}

We illustrate Theorem \ref{thm.semicont1} with some examples.

\begin{example}\label{ex.semicont1}
{\em
\begin{enumerate}
\item Let $\vp: A=\k [z] \to R=A[x,y]/ I$, $I=\langle x^2+y^2-z^2 \rangle \cap \langle x-y, y^2-z \rangle$. Then $X=V(I) = X^2 \cup X^1$, with $X^2$ the normal surface singularity defined by $x^2+y^2-z^2=0$ and  $X^1$ the smooth curve defined by $ x-y=y^2-z=0$, meeting $X^2$ at $(0,0,0)$ and 
$(x,y,z)=(x,y,2)$ with $x=y=\pm\sqrt 2$ (if $\sqrt 2 \in \k$). 

For $z=0$ the fiber $X_0$ is the nodal curve 
$x^2+y^2=0$ with an embedded point, and we compute $\eps_\k(X_0)=2$ and thus $\del_\k(X_0)=-1$. 
For $z=t\in \k, t\neq 0,2$ the fiber $X_t$ is a smooth curve 
and two extra reduced points not on the curve. Hence $X_t = (\oX)_t$ is normal with $\eps_\k(X_t)=2$. We get $\del_\k(X_t)=-2$ and $\del_\k(X_0)-\del_\k(X_t)=1$.

The normalization $\oX$ is the disjoint union of $X^{>1}=X^2$ and $X^1$. 
$(\oX)_0$ is the disjoint union of a nodal curve $(\oX^{>1})_0$ and a double point $(\oX^{1})_0$. We get
$\del_\k((\oX^{>1})_0) = 1$, confirming statement  (2) of Theorem \ref{thm.semicont1}.

Now let $\eta:=\langle 0\rangle$ be the generic point of $\Spec A$. The generic fiber is then $R\otimes_AQ(A)= \k(z)[x,y]/I$ cosisting of the regular curve 
$x^2+y^2-z^2=0$ (if char$(\k)\neq2$) and the isolated reduced point defined by the field $\k(z)[x,y]/\langle x-y, y^2-z \rangle =\k(z)[y]/\langle y^2-z \rangle$
a field extension of $\k(z)$ of degree 2.  Hence $\del_{\k(z)}(X_\eta)=-2.$

We get $\del_\k(X_0)-\del_{\k(z)}(X_\eta)=1$ and $\del_\k(X_t)-\del_{\k(z)}(X_\eta)=0$ if $t\neq0,2$ and $\del_\k((\oX^{>1})_t) = 0$. Both equalities confirm thus statement  (3) of Theorem \ref{thm.semicont1}.

Since $R$ is of finite typ over $A$ the flat map $f$ is open (in fact $f(X) = \Spec A$).  This is different in the following example.

%Consider $\vp: A=\Z \to R=\Z_{\langle p\rangle}[x,y,z]/I$, with $I$ as in (1)  and let $\fm$ a maximal ideal in $R$. If $p\in \fm$ then $\fm \cap A = \langle p\rangle$. 
% If $p\not\in \fm$ then $\fm \cap A = \langle 0\rangle$ (otherwise  $\fm \cap A = \langle q\rangle$ for some prime number $q \neq p$ and since $q$ is a unit in $R$, $q\not\in \fm$). It follows that $f=\Spec \vp: X=\Spec R \to S=\Spec A$ is flat and $f(X) = \{\langle p\rangle, \langle 0\rangle\}$ (every point in the image is the image of a closed point).
% 
% As an even simpler example take $A=\Z$ or $\k[t]$ and $R = Q(A)[x,y,z]$. Then $f(X) = \{\langle 0\rangle\}.$
% }
%\end{example}
%
%\begin{example}\label{ex.open}
\item Consider now $\vp: A=\k [z] \to R=\k [z]_{\langle z\rangle}[x,y,z]/I$, with $I$ as in (1).  Let $\fp$ be a prime ideal in $R$. If $z\in \fp$ then $\fp\cap A = \langle z\rangle$. 
 If $z\not\in \fp$ then $\fp \cap A = \langle 0\rangle$ (otherwise  $\fp \cap A = \langle p\rangle$ for some irreducible polynomial  $p(z) \not\in \langle z\rangle$ and since $p$ is a unit in $R$, $p\not\in \fp$). It follows that $f=\Spec \vp: X=\Spec R \to S=\Spec A$ is flat and $f(X)$ consists of two points $\langle z\rangle$ and $\eta=\langle 0\rangle$.
 As in (1) we get  $\del_\k(X_0)=-1$ and  $\del_\k(X_\eta)=-2.$
 
 We note that $f^{-1}(\eta) = \{\fp\in \Spec R \mid \fp \cap A = \langle 0\rangle\}
 $  and none of these $\fp$ is closed in $\Spec R$ ($R/\fp$ is not a field).  On the other hand, infinitely many of these prime ideals 
 (e.g.  $\fp=\langle x-p, y-q\rangle, p\in \k(z), q=(p-z)(p+z)$) are closed points of the fiber $X_\eta = \Spec \k(z)[x,y]/I $.
 \end{enumerate}
 }
\end{example}

\begin{example} \label{ex.nnormal2}
{\em We provide two examples showing the necessity of the assumptions  in 
Theorem \ref{thm.semicont1}.
\begin{enumerate}
\item The condition that $(\oX^{>1})_t, t\neq s$, is normal in 
Theorem \ref{thm.semicont1} (2) is necessary in positive characteristic (it automatically holds in characteristic 0, see Theorem \ref{thm.char0}):
Let $A=\k[z], \k$ algebraically closed, char$(\k) = p >2$, and $R = A[x,y]/\langle f\rangle, f = y^2 - x^p -z$. 
Then $R$ is regular, hence normal, of dimension 2 and the canonical map $\vp: A \to R$ is flat ($z$ is a non-zero divisor in $R$). 
For every closed point $\langle z-t\rangle, t\in \k,$ the fiber $R(t) = \k[x,y]/y^2-x^p -t$ has a reduced isolated non-normal point,
which is not  normal since $y^2 = (x+ t^{1/p})^p$. 

Since $\delta_\k(R(t)) = (p-1)/2$ is the same for every $t\in \k$ and since $\oR^{>1}(s) = R(s)$ the equality in Theorem \ref{thm.semicont1} (2) does not hold for $s, t \in \k$, while the equality in (3) holds for the generic fiber. 
Note that the generic fiber $\k(z)[x,y]/y^2-x^p -z$ is regular (hence normal) but not geometrically normal: it is not-normal in $\k(z^{1/p})[x,y]/y^2-x^p -z$.

This example shows also that $\del$ of the generic fiber is stricly smaller than $\del$ of every fiber over a closed point.
\smallskip

\item The assumption that $\Ker (\nu^{>1}: R\to \ol{R^{>1}})$
is finite over $A$ is necessary for the upper semicontinuity of $\eps$ even in characteristic 0 (the finiteness over $A$ of $\Ker(\nu)$ is not sufficient). For
$A=\k[x]$, $R = A[y]/\langle y(xy-1)\rangle$, we have $\eps_\k(R(0)) = 1$ but
$\eps_\k(R(s)) = 2$ for $s\in \k-\{0\}$.  In this case $R^{>1}=0$ and $X^{>1}=\emptyset$. Hence $\Ker(\nu^{>1})=R$, which is quasi-finite but not finite over $A$ (the class of $x$ in $R$ is not integral over $A$), while $\Ker(\nu) =0$.
See also Example \ref{ex.affine}.
\end{enumerate} }
\end{example}

Now let $X\to S$ be a morphism of Noetherian schemes with finite normalization map $n: \oX \to X$ and  $\wt{ C}_X$\footnote{\,If $X$ is covered by open affine sets  $X_i =\Spec R_i$, then $\wt{ C}_{X}|X_i = \wt{ C}_{R_i} = \Spec(R_i/\wt{\kc}_{R_i})$, c.f. Definition \ref{def.cond}.} 
the extended  conductor scheme, which is supported on the non-normal locus $\NNor(X)$ of $X$. 
Recall that $X^i$ resp. $X^{>i}$ denote the union of the irreducible components of $X^{red}$ of dimension  $i$ resp. $>i$. We say that {\em  $\NNor(X)$ and $X^1$ are finite over $S$} if $\ko_{\wt{ C}_X}$ and $\ko_{X^1}$ are  finite 
$\ko_S$-modules.
In view of Lemma \ref{lem.fin} this holds iff
$\Ker (n_*\ko_\oX \to \ko_{X^{>1}})$ and $\Coker (n_*\ko_\oX \to \ko_X)$ 
are finite $\ko_S$-modules.
With this notation and that of Theorem \ref{thm.semicont1} we get:

\begin{corollary}\label{cor.nagata}
Let  $f: X \to S$  be a flat morphism of schemes with $S$ the spectrum of a PID and $X$ Nagata. 
Assume that  $\NNor(X)$ and $X^1$ are finite over $S$ and that
for each $s\in f(X)$ the fiber $X_s $ has finitely many isolated non-normal singularities with residue fields finite over $k(s)$.
Then  the conclusions of Theorem \ref{thm.semicont1} hold.  

Morever, if all irreducible components of $X$ have dimension $\geq 2$ then the conclusions of Theorem \ref{thm.semicont1} hold with
$(X^{>1})_t = X_t$ and $(\ol{X^{>1}})_t = (\ol{X})_t$ for $t\in U$.
\end{corollary}

\begin{proof} We may assume that $X=\Spec R$ with $R$ a Nagata ring and  
$S= \Spec A$, $A$ a PID. Then
$\oR$ is finite over $R$ (\cite[Lemma 10.157.2, tag 03GH]{Stack}) and since $R/\fp R, \fp \in \Spec A$, is Nagata (cf. Remark \ref{rm.nagata}).
the normalization of $X_s$ is finite over $X_s$.  Now apply Theorem \ref{thm.semicont1}.
\end{proof}

The most interesting cases are perhaps when $R$ is (essential) of finite type over $A$, 
i.e., $R=A[x]/I$ with $x=(x_1,...,x_n)$ and $I$ an ideal, or $R$ is the localization $A[x]/I$ at some prime ideal. If $A$ is a PID which is Nagata (see Remark \ref{rm.nagata} for examples), then $R$ is Nagata by  Remark \ref{rm.nagata}.
Since a flat morphism is open in this situation (\cite [tag 01UA, Lemma 28.24.9.]{Stack}) we get the following corollary, which applies in particular to families of generically reduced curves (if $X$ is pure 2-dimensional).

\begin{corollary}\label{cor.nagata3}
Let  $f: X \to S$  be a flat morphism with $S$ the spectrum of a Nagata PID and $X$ locally (essentially) of finite type over $S$ and without 1-dimensional components. Assume that $\NNor(X)$ is finite over $S$ and that
each fiber $X_s $ has finitely many isolated non-normal singularities with residue fields finite over $k(s)$. Let  $n : \oX \to X$ be the normalization of $X$.

Then $\ol{ f} := f\circ n : \oX \to S$ is flat,  $\del_{k(s)}(X_s) < \infty, \,\eps_{k(s)}(X_s) < \infty$, and for each  $s\in S$ there exists an open neighbourhood $U\subset S$ of $s$ such that  the following holds (with
$(\ol{X})_t :=  \ol {f}^{-1}(t)$):
\begin{enumerate}
\item If $(\ol{X})_t $  is normal for $t \in U\sm\{s\}$, then\\
$ \delta_{k(s)}(X_s)-\delta_{k(t)}(X_{t}) 
=\delta_{k(s)}((\ol{X})_s)  \geq 0. $
\item $\delta_{k(s)}(X_s)-\delta_{k(\eta)}(X_{\eta}) =\delta_{k(s)}((\ol{X})_s) \geq 0$, $\eta$ the generic point of $S$.
\item $\varepsilon_{k(s)} (X_s)-\varepsilon_{k(t)}(X_{t})= \eps_{k(s)}(({X})_s)\geq 0$ for $t\in U\sm \{s\}$.
\item $(\ol{X})_t $  is reduced for every $t\in U$. If $X$ is reduced then $X_t$ is reduced for $t\in U\sm \{s\}$.\\
\end{enumerate}
\end{corollary}

Theorem \ref{thm.semicont1} and its corollaries say that over the spectrum of a PID the delta invariant of the generic fiber $X_\eta$ is minimal among all fibers $X_s, s\in f(X)$. It does not say, however, that the delta invariant of a special fiber $X_s$ is bigger or equal than the delta invariant of the fibers $X_t$ over closed points $t$ in a neighbourhood of $s$, except if $(\ol{X^{>1}})_t$ is normal for $t\neq s$. By Example \ref{ex.nnormal2} this may not be true and the equality in Theorem  \ref{thm.semicont1} (2) does not hold in positive characteristic (nevertheless, semicontinuity may hold in general but we do not know this). For characteristic 0 see Theorem \ref{thm.char0}, Section \ref{sec.4}. 

%Such a general semicontinuity statement holds if the fibers $\overline f^{-1}(t)$ are normal for $t\neq s$,  which is true for $X$ of finite type over a field of characteristic 0 (cf. Proposition \ref{prop.char0}) or for complex or real analytic morphisms.
%By Example \ref{ex.nnormal2} $\overline f^{-1}(t)$ may not be normal in positive characteristic.

%======================================================
\section{Fiberwise and Simultaneous Normalization}\label{sec.3}

While the notion of simultaneous normalization is well known, the following (weaker) definition of fiberwise normalization is new. It is useful if the
residue fields of the base scheme are not perfect.

\begin {definition} \label{def.simnorm}
Let $m: \wX \to X$ and $f:X \to S$ be morphisms of schemes. 
\begin {enumerate}
\item %A finite morphism $m: \wX \to X$  
We call  $m: \wX \to X$ a {\em fiberwise normalization} of $f$ if
\begin{enumerate}
\item $m$ is finite,
\item the composition $\wt{f} := f\circ m:  \wX \to S$ is flat, 
\item the non-empty fibers of $\wt{f} $ are normal, and
\item the induced map $m_t: \wt{f}^{-1}(t)  \to f^{-1}(t)$ is birational
%\,\footnote{\,A morphism of schemes is {\em birational} if it is a bijection between the generic points and an isomorphism of the corresponding local rings.}
for every $t\in f(X)$ 
\end{enumerate}
\item A fiberwise normalization of  $f$ is called a {\em  simultaneous normalization}  if the non-empty fibers of $\wt{f} $ are geometrically normal.
\end{enumerate}
\end{definition}

Recall that a $\k$-algebra $R$ is called {\em geometrically normal} (resp.  {\em geometrically reduced}) if $R\otimes_\k\k'$
is normal (resp. reduced) for every  field extension $\k \subset \k'$ (equivalently, for every finite field extension). If $\k$ is a perfect field, then a $\k$-algebra is normal (resp. reduced) iff it is geometrically normal (resp. reduced), see \cite[Lemma 10.43.3.,tag 030V and Lemma 10.160.1.,tag 037Z]{Stack}.
A morphism of rings $\vp: A\to R$ is  called {\em normal} (resp. {\em reduced}) if it is flat  and if the non-empty fibers $R(\fp):=R\otimes_A k(\fp)$,  are geometrically normal (resp. geometrically reduced) as $k(\fp)$-algebras, where $k(\fp)=A_\fp/\fp A_\fp = \Quot(A/\fp)$ is the residue field of $A_\fp$, $\fp \in \Spec A$.

A morphism of schemes $f:X \to S$ is geometrically normal (resp. geometrically reduced), if this holds for the induced morphisms of local rings. 
Hence, if the residue fields of all local rings of $S$ are perfect (e.g. of characteristic 0), then the notions of fiberwise normalization and simultaneous normalization coincide. 

Note that simultaneous normalization is preserved under base change,  while this is in general not the case for fiberwise normalization of schemes over non-perfect fields. On the other hand, the following results 
%(e.g. Lemma \ref{lem.base} and Theorem \ref{thm.simnormPID}) 
show that the weaker assumption of a fiberwise normalization is often sufficient and useful.
\medskip

\begin{lemma} \label{lem.base}
Let  $f:X \to S$ be a flat morphism of schemes and assume that $f$ admits a fiberwise normalization  $m: \wX \to X$. Let $\wt{f} = f\circ m:  \wX \to S$. Then the following holds:
\begin{enumerate}
\item $m$ is birational. 
\item $\wX$ is reduced (resp. normal) at $\tilde x$ if and only if $S$ is reduced (resp. normal) at $\wt{f}(\tilde x)$. If $\wX$ is normal, then $\wX \cong \oX$ and  $m$ is the normalization map.
\item The induced fiber map $m_t: \wX_t = \wt{f}^{-1}(t)  \to f^{-1}(t)= X_t$  is the normalization of $X_t$ for every $t\in f(X)$.
\item Let $S$ be normal, $t\in S$ and $x\in X_t$. Denote by $X^i_t, i = 1,...,r,$ the irreducible components of $X_t$ passing through $x$ and by $X^j, j = 1,...,s,$ the irreducible components of $X$ passing through $x$. Then $r=s$ and  for each $j$ there exists a unique $i=i(j)$ such that $X^i_t \subset X^j$. The corresponding components satisfy the {\em dimension formula}
$$\dim (X^j,x) = \dim (X^i_t,x) + \dim (S,t).\footnote{$\dim (X,x)$ denotes the dimension of the local ring $\ko_{X,x}$ of the scheme $X$ at $x$.} $$
In particular, if $\dim(X_t, x) > 0$ then each irreducible component of 
$X$ passing through $x$ has dimension $> \dim(S,t).$
\end{enumerate}
\end{lemma}

The dimension formula in (4) is not a direct consequence of the flatness of $f$, since the restriction of a flat map to an irreducible component need not be flat; it is a consequence of the existence of a fiberwise normalization.

%\GMG{If $S$ is irreducible and $t=\eta$ the generic point of $S$, then $\dim (X^j,x) = \dim (X^i_t,x)$ and the generic fiber of $f$ coincides with $X$, as a topological space. Kann das sein???}

\begin{proof} 
(1) A flat morphism $f$ maps a generic point $x \in X$ to a generic point $f(x)\in S$, since flat mappings have the going down property (\cite{Stack} Lemma 10.38.19, tag  00HS).
%since the induced mapping of local rings is faithfully flat and the local ring $\ko_{X,x}$ is a field  (the quotient field of the ring of the corresponding irreducible component of $X$). 
Since $m_t: \wt{f}^{-1}(t)  \to f^{-1}(t), t\in f(X)$, induces a bijection of generic points of the fibers and an isomorphism of their local rings, this holds also for 
$m$ by the following Lemma \ref{lem.birat} (see also \cite[Theorem 2.3]{CL06}), since the minimal primes of a ring correspond uniquely to the generic points of its spectrum.
%satisfies the dimension formula  $\dim (X,x) = \dim (X_{f(x)},x) + \dim (S,f(x))$
 
(2) The first statement follows from  \cite[Corollary of Theorem 23.9]{Mat86}. The second statement follows since $m$ is finite by definition and birational by (1).

(3) Since $m_t$ is finite and birational it is the normalization map.

(4) Since $S$ is normal, $m$ is the normalization map by (2) and thus the number $r$ of irreducible components $X$ passing through $x$ equals $\# m^{-1}(x)$. In the same way  $s= \# m_t^{-1}(x)$ holds. Since  
$\# m^{-1}(x) =\# m_t^{-1}(x)$ we get the first statement of (2).
For each $X^j$ there exists a unique point $\tilde x^j \in m^{-1}(x)\cap m^{-1}(X^j)$ and we get for $i=i(j)$
$$\dim (X^j,x) = \dim(\wX, \tilde x^j) = \dim(\wX_t, \tilde x^j) + \dim(S, t) = \dim (X^i_t,x)  + \dim(S, t),$$
since $\wt{f}$ is flat.
Finally, $\dim(X_t,x) > 0$ means of course that all irreducible components $X^i_t$ have dimension $> 0$ at $x$ and the last statement follows.
\end{proof}

\begin{lemma}\label{lem.birat}
Let $A$ be a ring and $R$ a flat $A$-algebra. Then the total ring of fractions satisfies $Q(R)=Q(R\otimes_AQ(A))$. 
Let $\wR$ be a flat $A$-algebra
and $\mu: R\to \wR$ an $A$-algebra homomorphism
inducing an isomorphism $R\otimes_AQ(A)\cong \wR \otimes_AQ(A)$. Then $\mu$ induces an isomorphism $Q(R)\cong Q(\wR)$
and $\mu: R\to \wR$ is birational. 
\end{lemma}

\begin{proof} Let $R$ be an $A$-algebra via $\vp:A\to R$. 
 $R\otimes_AQ(A)=\{\frac{r}{\vp(a)} | r\in R \text{ and }a\in A \text{ a non-zerodivisor}\}$. Since $\vp$ is flat,  it maps non-zero divisors to non-zero divisors. Hence $R\otimes_AQ(A) \subset Q(R)$ and
 $Q(R)=Q(R\otimes_AQ(A))$. Similarly we obtain
  $Q(\wR)=Q(\wR\otimes_AQ(A))$. % \GMG{to do}
  
By assumption $\mu$ induces an isomorphism $R\otimes_AQ(A)\cong \wR \otimes_AQ(A)$ and it follows that $\mu$ induces an isomorphism $Q(R)\cong Q(\wR)$. 
Let
 $\frak{p}_1,\ldots,\frak{p}_r$ be the associated prime ideals of $R$.
 Then $\fp_1\cup ... \cup\fp_r$ is the set of zero divisors of $R$ and
 $\frak{p}_1Q(R),\ldots,\frak{p}_rQ(R)$ are the prime ideals of $Q(R)$
 (\cite[Theorem 4.1]{Mat86}). Here the minimal associated primes of $R$ correspond to the minimal primes of $Q(R)$ (also the embedded associated primes of $R$ correspond to the embedded primes of $Q(R)$, but this is not relevant for us),  and similarly for  $Q(\wR)$.
 Since $Q(R)\cong Q(\wR)$ the minimal prime ideals of $Q(R)$ and $Q(\wR)$  are in 1-1 correspondence. Since for a minimal prime $\frak{p}$ of $R$ we have $R_{\frak{p}}=Q(R)_{\frak{p}Q(R)}$ we obtain all together a bijection between the minimal primes of $R$ and $\wR$ and an isomorphism of the corresponding local rings. This implies that $\mu:R\longrightarrow \tilde R$ is birational.  \end{proof}

We want to characterize fiberwise resp. simultaneous normalization of a family of INNS numerically by the $\del$-invariant of the fibers. The following example shows that we have to be careful with families of affine fibers.

\begin{example} \label{ex.ker}
{\em Consider Example \ref{ex.nnormal2} (2) with
$R = A[y]/\langle y(xy-1)\rangle$ and $A=\k[x]$.  The canonical map $A\to R$ is flat, with normal fibers, and hence the identity $R\to R$ is a fiberwise normalization (resp. simultaneous normalization if char$(\k)=0$). But the $\del$-invariant is not constant ($\del_\k(R(0)) = -1$ and $\del_\k(R(s)) = -2$ for $s\in \k-\{0\}$). The same holds for  $A=\k[x]_{\x}$ and $s$ the generic point of $\Spec A$. This does not contradict the following Theorem \ref{thm.simnormPID} since $\Ker(\nu^{>1})$ is not finite over $A$. For another example with all fibers reduced curves, see Example \ref{ex.affine}.
}
\end{example}

The following theorem resp. its corollary is a numerical characterization for the existence of a fiberwise  normalization of a family of isolated normal singularities over the spectrum of a PID.

\begin{theorem}\label{thm.simnormPID}
Let  $\vp: A \to R$  be a flat morphism of rings with $A$ a principal ideal domain and $R$ Nagata. Let  $\nu :R \to \oR$ be the normalization. Set  $S=\Spec A$, $X =\Spec R$,  $\oX =\Spec \oR$,
$f = \Spec \vp: X \to S$, $n =\Spec \nu :\oX \to X$.
Assume that
$\Coker(\nu)$ and $\Ker (\nu^{>1}: R\to \ol{R^{>1}})$ are finite over $A$ and that
for each $t\in f(X)$ the fiber $X_t $ has finitely many isolated non-normal singularities with residue fields finite over $k(t)$.
\begin{enumerate}
\item Assume that $f$ admits a fiberwise normalization. Then 
$\del_{k(t)}(X_t)$ is constant on $S$. %Moreover, if $f(X)$ contains a closed point of $S$ then $X^1$ is disjoint from $X^{>1}$.
% and the 1-dimensional part $X^1$ of $X$ is regular and does not meet the higher dimensional part $X^{>1}$.
\item 
$\del_{k(t)}(X_t)$ is  constant on $S$ if and only if $f^{>1}:X^{>1} \to S$ admits a fiberwise normalization.
\item If  $X$ has no 1-dimensional irreducible components then $f$ admits a fiberwise normalization if and only if $\del_{k(t)}(X_t)$ is  constant on $S$.
\end{enumerate}
\end{theorem}

\begin{proof}
(1) Since $A$ is regular, hence normal, any fiberwise normalization is the normalization by Lemma \ref{lem.base}. Since the fibers of $\ol{f}=f\circ n$ are  normal by assumption,
  $(\ol{X}^{>1})_t$ is normal and has positive dimension. Hence $\del_{k(t)} ((\oX^{>1})_t) =0$ for $t\in S$. The constancy of $\del_{k(t)}(X_t)$ follows from Theorem \ref{thm.semicont1} (2) (resp. Corollary \ref{cor.nagata}).
 
(2) If $\del_{k(t)}(X_t)$ is constant then $\del_{k(t)}(X_t) = \del_{k(\eta)}(X_\eta)$ for $t\in \Spec A$ and $\eta$ the generic point of $\Spec A$.  By Theorem \ref{thm.semicont1} (3) (applied to any $t\in U$) 
$\del_{k(t)}((\ol{X^{>1}})_t) =0$ and $(\ol{X^{>1}})_t$ is reduced. Hence 
$(\ol{X^{>1}})_t$ is normal and equal to the normalization of $(X^{>1})_t$ (since $(\ol{X^{>1}})_t \to (X^{>1})_t$ is  birational and finite).
Since $\ol{f}$ is flat by Lemma \ref{lem.flat}, $\ol{f}^{>1}= f| \ol{X}^{>1}$ is also flat and $n^{>1}:\oX^{>1} \to X^{>1}$
is a fiberwise normalization of $f^{>1}$. 

By Theorem \ref{thm.semicont1} (1) $\del_{k(t)}(X_t)$ is constant 
$\iff$ $\del_{k(t)}((X^{red})_t)$ is constant $\iff$ $\del_{k(t)}((X^{>1})_t)$ is constant. If  $f^{>1}$ admits a fiberwise normalization, then $\del_{k(t)}((X^{>1})_t)$ is constant by (1) and thus $\del_{k(t)}(X_t)$ is constant.

(3)  The "only if " direction follows from (1). 
$X^1 =\emptyset$  means  $r_1(X)=0$ and $f^{>1}= f^{red} :X^{red} \to S$. By (2) $\del_{k(t)}(X_t)$ = constant implies that $\oX \to X^{red}$ is a
fiberwise normalization of $f^{red}$. But then  $\oX \to X$ is a fiberwise normalization of $f$ (the flatness of $f$ implies the flatness of  $f^{red}$ and
$\ol{f}: \oX \to X \to S$ by Lemma \ref{lem.flat}).
\end{proof}

\begin{corollary}\label{cor.simnormPID}
In addition to the assumptions of Theorem \ref{thm.simnormPID} assume that all residue fields $k(\fp)$, $\fp \in \Spec A$, are perfect (e.g., $A=\Z$ or $A$ a $\k$-algebra with $char (\k) =0$). Then the statements (2) and (3) of Theorem \ref{thm.simnormPID} hold with 'fiberwise normalization' replaced by 'simultaneous normalization'.
\end{corollary}

We conjecture that Theorem \ref{thm.simnormPID} and Corollary 
\ref{cor.simnormPID} hold if $A$ is normal. However, at the moment we don't have a proof as the non-reducedness of the fibers creates severe technical difficulties for non-PIDs as base rings.

%==================================================
\section{Characteristic 0 and the Analytic Case}\label{sec.4}

If the characteristic is zero then the assumptions of the semicontinuity theorem (Theorem \ref{thm.semicont1}) can be weakened and the statement is stronger for morphisms of finite type and for analytic morphisms (using theorems of Bertini and Sard type). The main property is that reduced spaces are regular (hence normal) on an open dense subset.
\begin{theorem} \label{thm.char0}
Let $\k$ be a field of characteristic zero and let $\vp:A\to R$ be a flat morphism of $\k$-algebras.  Assume that $A$ is a PID and that $R$ is of finite type over $\k$. 
Let $\nu :R \to \oR$ be the normalization and assume that   $\Ker(\nu^{>1})$ and $\Coker(\nu)$ are finite over $A$. 
Let $X =\Spec R$,  $\oX =\Spec \oR$,
$n =\Spec \nu :\oX \to X$, $f =\Spec \vp : X \to S=\Spec A$ and
$\ol{f} =f\circ n : \oX \to S$.
Let $X_s =f^{-1}(s)$, $s\in f(X),$ have only finitely many non-normal points. %\GMG{'with residue fields  finite over $k(s)$' muss man nicht voraussetzen,  da $R$ von endlichem Typ ueber $\k$ ist, s.u im Beweis OK?}

Then  $\eps_{k(s)}(X_s)$ and  $\del_{k(s)}(X_s)$ are finite for $s\in S$ and
 each $s\in S$ has an open neighbourhood $U$ such that 
 for $t \in U\sm\{s\}$   the following holds:
\begin {enumerate}
\item $(\oX^{>1})_s $ is reduced and $(\oX)_t$ is regular, hence normal.
\item $\delta_{k(s)}(X_s)-\delta_{k(t)}(X_{t}) =\delta_{k(s)}((\ol{X^{>1}})_s)\geq 0$,
\item $\eps_{k(s)}(X_s)-\eps_{k(t)}(X_t) =\eps_{k(s)}(({X}^{>1})_s)\geq 0.$
\end{enumerate}
\end{theorem}

\begin{proof}
Since $X$ is of finite type over $\k$,  any fiber $X_s$, $s \in f(X)$, is of finite type over $k(s)$. Since the non-normal points $x$ of $X_s$ are closed in $X_s$, the field extension  $k(s) \subset k(x)$ is finite (by Hilbert's Nullstellensatz). Moreover, X
is Nagata by Remark \ref{rm.nagata} and hence $\oX$ is finite over $X$, as well as the normalization  $\ol{(X_s)}$ 
over $X_s$. 
From Theorem \ref{thm.semicont1} (4) and (5) we get that statement  (3) holds and that $(\oX^{>1})_s $ is reduced for $s\in f(X)$, while statement (2) follows from  Theorem \ref{thm.semicont1} (2) together with the regularity of $(\oX)_t, t\neq s$ .

It remains to show that  $(\oX)_t$ is regular for $t\neq s$ in some neighbourhood of $s$.
Since $f$ is flat, the map $\ol{f}:\oX \to S$ and its restriction to any irreducible component of $\oX$  is flat by Lemma \ref{lem.flat} and dominant ($A\to \oR$ is injective).
Since $\oX$ is reduced, there exists an open dense subset $V\subset \oX$ such that $V$ is smooth (\cite[Theorem 21.3.5]{Va17}) and, since $\k$ is perfect, a point $x$ is smooth iff its local ring is regular (\cite[Lemma 32.25.8.,tag 0B8X]{Stack}). 
 Since $\overline f$ is flat and of finite presentation, it follows, that  $\ol{f} (V)$ is an open subset of $S$ (\cite[Proposition 10.40.8.,tag 00I1]{Stack}). Since it is not empty, it is dense.
% 
%Since $\oX_t$ is reduced by Theorem \ref{thm.semicont1}, the smooth locus of $\oX_t$ is open and dense. If $x\in \oX$ is a smooth point of  $\oX_t$, it is a smooth point of  $\oX$ (\cite[Theorem 23.7]{Mat86}) since $\ol{f}$ is flat by Lemma \ref{lem.flat} (1)(ii). It follows, that  $\ol{f} (V)$ is an open dense subset of $S$ (\cite[tag 00I1, Proposition 10.40.8.]{Stack}). 

By (\cite[Theorem 25.3.3]{Va17}) there exists an open dense subset $W\subset S$ such that the restriction 
of $\ol{f}$ to the non-empty open set $V\cap \ol{f}^{-1}(W)$ is a smooth morphism. Then $X_t$ is smooth, hence regular for $t\in W$. 
Since $S$ is 1-dimensional, $W$ is the complement of finitely many points,
this implies (1). 
\end{proof}

Let us consider the real and complex analytic case with $\K \in \{\R, \C\}$ and
$$\delta:=\delta_\K, \quad \eps:=\eps_\K.$$ 
 The following theorems were proved in \cite{Gr17} for $\K=\C$. The proofs follow similar arguments as in Section \ref{sec.2} and  \ref{sec.3}. The proofs
  for $\K= \R$ are analogous. 
  
  The analytic case is in some sense technically easier than the general algebraic case: complex spaces are Nagata of characteristic 0, and all points are closed. Moreover, the non-normal locus $\NNor(f)$ of an analytic morphism $f:X\to S$ is a closed analytic subset of $X$ and $f(\NNor(f))$ is neglectable\,\footnote{\, i.e. contained in a countable union of nowhere dense locally closed analytic subsets of S.} in $S$, provided there is an open dense subset $V \subset S$ consisting of smooth points of $S$ such that $f^{-1}(V)$ consists of normal points of $X$. If the restriction of $f$ to $\NNor(f)$ is proper, then "neglectable" can be replaced by "a nowhere dense closed analytic subset" (see \cite[Theorem 2.1(3)]{BF93} for a proof of these statements).
  
For $f:X\to S$ and $t\in S$, if the fiber $X_t =f^{-1}(t)$ has only finitely many isolated non-normal singularities $x_1,...,x_r$, we use the notaion $$\del(X_t)=\sum_i^r\del(X_t,x_i).$$

\begin{theorem}{\em(\cite[Theorem 7.14]{Gr17})}\label{thm.semicont-an}
Let $f:(X,x)\to (\K, 0)$ be flat morphism of \,$\K$-analytic germs with fibre $(X_0, x)$ an isolated non--normal singularity. If $f:X \to T$ is a sufficiently small representative, then $X_t$ has only finitely many INNS and the following holds  for $t\neq 0$.  
\begin{enumerate}
\item $(\oX^{>1})_0$ is reduced and $(\oX)_t$ is smooth.
\item  $\delta(X_0)-\delta(X_t)\\ 
\text{  } =  \delta((X^{red})_0)-\delta((X^{red})_t) \text{ with } (X^{red})_t = (X_t)^{red},\\
\text{  } = \delta((\overline{X})_0)-\delta((\overline{X})_t) \text{ with }  -\delta((\overline{X})_t)= \dim_\K \ko_{(X^1)_t},\\
\text{  } =\delta((X^{>1})_0)-\delta((X^{>1})_t)\\
\text{  } =\delta((\overline{X}^{>1})_0)\geq 0$.

%\item $\delta(X_0)-\delta(X_t)$ is equal to
%	\begin{enumerate}
%	\item [(i)] $\delta((X^{red})_0)-\delta((X^{red})_t)=\delta((X^{red})_0)-\delta((X_t)^{red})$,
%	\item [(ii)] $\delta((\overline{X})_0)-\delta((\overline{X})_t)=\delta((\overline{X})_0)+ \dim_\K \ko_{(X^1)_t}$,
%	\item [(iii)] $\delta((X^{>1})_0)-\delta((X^{>1})_t)=\delta((\overline{X}^{>1})_0)\geq 0$,
%\end{enumerate}
\item  $\varepsilon (X_0)-\varepsilon(X_t)=\varepsilon((X^{>1})_0)\geq 0$.
	\end{enumerate}
Moreover, if $r_1(X,x)=0$ (i.e., $(X,x)$ has no 1-dimensional components, or, equivalently, $X_t$ has no isolated points for $t\ne0$), then
\begin{enumerate}
\item $(\oX)_0$ is reduced and $(\oX)_t$ is smooth.
\item $\delta(X_0)-\delta(X_t)= \delta((\overline{X})_0)\geq 0$.
\item  $\varepsilon (X_0)-\varepsilon(X_t)=\varepsilon((X^{red})_0)\geq 0$.
\end{enumerate}
In particular, if $X$ is reduced, then $\varepsilon(X_t)=0$ and hence $X_t$ is reduced for $t\neq 0$.
\end{theorem} 

Note that for morphisms of  analytic germs we do not need to assume that   $\Ker(\nu^{>1})$ and $\Coker(\nu)$ are finite over $(\K,0)$. This follows already from the assumption that $(X_0, x)$ is an isolated non--normal singularity. Then $\Ker(\nu^{>1})$ and $\Coker(\nu)$ are quasi-finite over $(\K,0)$ and thus have a representative (in the Euclidean topology) which is finite over $\K$.  

For $\K = \C$ we have $\dim_\K \ko_{(X^1)_t} =\# \{$isolated points of $X_t \}$.  Statement (2) is in \cite[Theorem 7.14]{Gr17} formulated with $r_1(X)$ instead of $\dim_\K \ko_{(X^1)_t}$, which is wrong in general. We have $r_1(X)\leq \dim_\K \ko_{(X^1)_t} $ and $r_1(X)=0$ iff $\dim_\K \ko_{(X^1)_t} =0$. 

\begin{theorem} {\em(\cite[Theorem 7.17]{Gr17})} \label{theorem1} Let $f : (X,x) \to (\K,0)$ be flat with $(X_0,x)$ an INNS of dimension $\geq 1$.
Then there exists a small representative $f:X\to S$ of the germ $f$, such that
$f:X\to S$ 	admits a simultaneous normalization if and only if $\delta (X_t)$ is constant and $r_1(X,x) = 0$.
\end{theorem}

The assumption that $(X_0,x)$ an INNS implies that the non-normal locus of
$f$ (which is analytic in $X$) is finite over $S$ (for $S$ sufficiently small), 
and hence $X_t$ has only isolated non-normal singularities for $t\in S$.
We recall that an excellent local ring $(R,\fm)$ is normal if $R/f$ is normal for $f\in \fm$ a non-zero divisor of $R$ (\cite[Lemma 4.4]{BF93}). Hence, for $f:X\to S$ a flat analytic morphism, with $S$ a smooth 1-dimensional manifold, the non-normal locus of $X$ is contained in the non-normal locus of $f$. 
Let $n:\oX \to X$ be the normalization of $X$. Then the non-normal locus of $X$ is the union of the supports of the sheaves $\Coker (n_*: \ko_X \to n_*\ko_\oX)$ and $\Ker (n_* : \ko_X \to n_*\ko_{\oX})$ and the assumption that $(X_0,x)$ an INNS implies that these sheaves are finite over $S$.
The necessity of $r_1(X,x) = 0$ follows from the dimension formula of the analytic analog of Lemma \ref{lem.base} (4). 
We thus see that analogous assumptions as in Theorem \ref{thm.char0} hold also in the analytic situation of Theorem \ref{thm.semicont-an}.
\medskip

It is not difficult to see that $\Coker (n_*: \ko_X \to n_*\ko_\oX)$ and $\Ker (n^{>1}_* : \ko_X \to n_*\ko_{\oX^{>1}})$ are finite over $S$ iff the non-normal loci of $X$ and $X^{>1}$, defined by the extended conductor schemes, are finite over $S$ (see Remark \ref{rem.semicont1} (1)).
 We use this in the following Theorem \ref{thm.global}
for (global) morphisms of analytic spaces.

\begin{theorem}{\em(\cite[Theorem 7.19]{Gr17})} \label{thm.global}
	Let $f: X\to S$ be a flat morphism of $\K$-analytic spaces with $S$ a $1$--dimensional analytic manifold. Assume that the non-normal loci of 
$f$ and of $f^{>1}: X^{>1} \to S$ are finite over $S$.
Then, for $t\in S$, the fiber $X_t $ has only finitely many isolated non-normal singularities  and the following holds:
	\begin{enumerate}  		
	\item $\del(X_s)-\del(X_t) =\delta((\overline{X}^{>1})_0)\geq 0.$	
	\item The following are equivalent
	\begin{enumerate}
		\item [(i)] $f$ admits a simultaneous normalization.
		\item [(ii)] $\delta(X_t)$ is locally constant on $S$ and the $1$--dimensional part $X^1$ of $X$ is smooth and does not meet the higher dimensional part $X^{>1}$.
	\end{enumerate}
	\end{enumerate}
In particular, if $X$ has no 1-dimensional part, then
\begin{enumerate}  	
\item	 $\del(X_s)-\del(X_t) =\delta((\overline{X})_0)\geq 0.$	
\item$f \text{ admits a simultaneous normalization} \iff \delta(X_t) \text{ is locally constant on }S.$
\end{enumerate}
\end{theorem}

The finiteness of $\NNor (f^{>1})$ over $S$ was mistakenly omitted in \cite[Theorem 7.19]{Gr17}.
The following example shows that it is necessary.

\begin{example}\label{ex.affine}
{\em Let $X= V(((tx-1)^2-y^3t^2)y)$, a reduced surface in $\C^3$, and $f:X\to \C$ the projection $(t,x,y)\mapsto t$. The non-normal locus of $f$ is the curve $\{y=tx-1=0\}$ which is quasi-finite over $\C$ (but not finite) and all fibers $X_t$ are reduced curves. 
It is easy to see that $\delta(X_0)=0$ but $\delta(X_t)=5$ for $t\neq 0$  (or we can use {\sc Singular} \cite{DGPS}), showing that $\del$ is not semicontinuous.
}
\end{example}

%======================================================

%%%%%%%%%%%%%%%%%%%%%%%%%%%%%%%%%%%%%

\end{document}